\documentclass[11pt]{amsart} 
\usepackage{verbatim, latexsym, amssymb, amsmath,color,}
\usepackage{epsfig}
\usepackage[linktoc=all]{hyperref}
\hypersetup{
	colorlinks,
	citecolor=black,
	filecolor=black,
	linkcolor=black,
	urlcolor=black
}

\def\R{\mathbb R}
\def\H{\mathbb H}

\def\N{\mathbb N}

\def\C{\mathbb C}

\newtheorem{thm}{Theorem}[section]
\newtheorem{lemm}[thm]{Lemma}
\newtheorem{cor}[thm]{Corollary}
\newtheorem{prop}[thm]{Proposition}
\theoremstyle{remark}

\theoremstyle{definition}

\title{Minimal surface entropy and average area ratio}
\author{Ben Lowe and Andr\'e Neves}
\address{Princeton University \\ Fine Hall \\ Princeton NJ 08544 \\ USA}
\email{benl@princeton.edu} 
\address{University of Chicago \\ Department of Mathematics \\ Chicago IL 60637\\ USA}
\email{aneves@math.uchicago.edu}
\thanks{The second author is partly supported by NSF  DMS-2005468 and a Simons Investigator Grant.}

\begin{document}

	\begin{abstract}
		On any closed hyperbolizable $3$-manifold, we find a sharp relation  between the minimal surface entropy (introduced by Calegari-Marques-Neves) and the average area ratio  (introduced by Gromov), and we show that, among metrics $g$ with  scalar curvature  greater than or equal to $-6$, the former is maximized by the hyperbolic metric. One corollary is to solve a conjecture of Gromov regarding the average area ratio.
		
		{Our proofs} use Ricci flow with surgery and laminar measures invariant under a $\text {PSL}(2,\R)$-action.
	\end{abstract}
	
	\maketitle 
	\section{Introduction}
	
	The interplay between scalar curvature, area, and topology is a beautiful chapter in mathematics.  For an extended overview of the subject containing the most recent developments, the reader can consult  Gromov's Four Lectures on Scalar Curvature  \cite{gromov4}. 
	
	We are interested in studying the area functional on closed  manifolds admitting a hyperbolic metric. Variational methods, which were pioneered by Schoen and Yau, are less effective in this setting  because  the restrictions imposed by the second variation of area are not sharp in the hyperbolic case. On the other hand, the study of length, curvature, and topology has a very rich literature in the negative curvature case, partly because the fact that the geodesic flow is Anosov brings an extra structure to the problem. In the same vein, there is an extra structure for minimal surfaces coming from a ``natural'' $\text {PSL}(2,\R)$-action on the space of minimal  immersions {(formalized very clearly by Labourie in \cite{Labourie}}). The general principle we follow, initiated by Calegari--Marques--Neves in \cite{cal-marques-neves}, is to combine the rigidity of that action (due to Ratner and Shah) with  geometric methods to obtain sharp relations between area, scalar curvature, and minimal surfaces.  One consequence is to answer a  conjecture of Gromov regarding the least possible value for the average area ratio.

	Using spin methods,  Min-Oo \cite{min-oo} proved a rigidity theorem for asymptotically hyperbolic  manifolds (see also \cite{CH,Wang}) and Ono \cite{ono},  Davaux \cite{davaux} proved sharp spectral inequalities for hyperbolizable manifolds.  Andersson, Cai, and Galloway \cite{andersson-cai-galloway} proved a  positive mass theorem for asymptotically hyperbolic  manifolds using variational methods.
	
	Let $(M,g)$ be a closed Riemannian orientable $3$-manifold admitting a hyperbolic metric $g_0$.  We refer the reader to   Section \ref{prelim} for all the definitions.
	
	\subsection{Minimal Surface Entropy}  
	Let $S_{\varepsilon}(M)$ denote the set of all homotopy class $\Pi$ of essential surfaces whose limit set is a $(1+\varepsilon)$-quasicircle. An important result of  Kahn--Markovic \cite{kahn-markovic2, kahn-markovic} establishes that $S_{\varepsilon}(M)\neq\emptyset$  (assuming $\varepsilon>0$) and provides an estimate for the cardinality of $S_{\varepsilon}(M)$. We define
	$$\text{area}_{g}(\Pi):=\inf \{\text{area}_{g}(S): S\in \Pi\}.$$
	Inspired by  the following expression for the volume entropy $E_{vol}(g)$ on negatively curved manifolds
	\begin{equation*}
		E_{vol}(g)=\lim_{L\to\infty}\frac{\ln\#\{{\rm length}_g(\gamma)\leq L:\,\gamma\text{ closed geodesic in }(M,g)\}}{L}
	\end{equation*}
	Calegari, Marques, and Neves defined in \cite{cal-marques-neves} the {\em minimal surface entropy}
	$$
		E(g):=\lim_{\varepsilon\to 0}\limsup_{L\to\infty}\frac{\ln \#\{\text{area}_g(\Pi)\leq 4\pi(L-1):\Pi\in S_{\varepsilon}(M)\}}{L\ln L}.
	$$
	The authors showed in \cite{cal-marques-neves} that $E(g_0)=2$ and that, among metrics $g$ with sectional curvature less than or equal to $-1$, $E(g)\geq E(g_0)=2$ and equality implies $g=g_0$ (up to isometry). The inequality follows almost directly from Gauss equation. On the other hand, to show that the hyperbolic metric is the unique minimizer the authors had to use  the rigidity of $\text {PSL}(2,\R)$-orbits in the frame bundle (due to Ratner  \cite{ratner} and Shah \cite{shah}).
	
	{The next entropy $E_{\bar\mu}(g)$  was introduced in \cite{marques-neves-conf} and counts only the homotopy classes for which the area-minimizing surfaces with respect to $g_0$ become equidistributed.

	For sufficiently small $\varepsilon$, each $\Pi\in  S_{\varepsilon}(M)$ induces  a  unit Radon measure $\mu_\Pi$ on the frame bundle $F(M)$ which is obtained by integration over the unique least area surface in $\Pi$ with respect to $g_0$ (see \eqref{pi.measure}). The homogeneous measure on $F(M)$ is denoted by $\bar\mu$ and called the Liouville measure.
	
	Let $\rho$ be a metric  which topologizes   the space of unit Radon measures on $F(M)$ and set ${S}_{\varepsilon,\bar\mu}(M)=\{\Pi\in S_\varepsilon(M): \rho(\nu_\Pi,\mu)<\varepsilon\}$. In \cite{marques-neves-conf} the {\em minimal surface entropy $E_{\bar\mu}(g)$ for $\bar\mu$}  of $(M,g)$  was defined as 
\begin{equation}\label{asymp.area}
E_{\bar\mu}(g)=\lim_{\varepsilon\to 0}\limsup_{L\to\infty} \frac{\ln \#\{\text{area}_g(\Pi)\leq 4\pi(L-1):\Pi\in S_{\varepsilon,\bar\mu}(M)\}}{L\ln L}.
\end{equation} 
Since the space of unit Radon measures on $F(M)$ is compact, the definition of $E_{\bar\mu}(g)$ does not depend on the metric $\rho$.  At the end of Section 2 in \cite{marques-neves-conf} it is shown that $E(g)$ and $E_{\bar\mu}(g)$   do not change if one uses $\liminf$ instead of $\limsup$.  

In  \cite{marques-neves-conf} the authors prove several sharp inequalities relating $E_{\bar\mu}(g)$  with the Liouville entropy. Some of the properties proven there are
\begin{itemize}
\item   $E_{\bar\mu}(g_0)=2$ and $E_{\bar\mu}(g)\leq E(g)$ for every metric $g$;
\item $E_{\bar\mu}(g)=E(g)$ for every metric $g$ on $M$ if $(M,g_0)$ contains no closed totally geodesic surfaces if (\cite[Proposition 5.1]{marques-neves-conf});
\item If $(M,g_0)$ contains a closed totally geodesic surface, there is $g$ near  $g_0$ for which $E_{\bar\mu}(g)<E(g)$ (follows from Corollary 1.1 with Theorem 1.2 of  \cite{marques-neves-conf}).
\end{itemize}

	\subsection{Average Area Ratio} Consider another closed hyperbolic $3$-manifold  $N$ with an hyperbolic metric $h_0$. The Grassmanian bundle over $N$ (or $M$) of unoriented  tangent $2$-planes  to $N$ (or $M$) is denoted by $Gr_2(N)$ (or  $Gr_2(M)$).  The metric $h_0$ induces a natural metric on $Gr_2(N)$.

	Given $T:M\rightarrow N$   a smooth map, the following function is defined for almost all $(x,P)\in Gr_2(N)$:
	$$(x,P)\in Gr_2(N)\,\mapsto \,|\wedge^2T^{-1}|(x,P):= \lim_{\delta\to 0}\frac{{\rm area}_g(T^{-1}(P_{\delta}))}{\delta},$$
	where $P_{\delta}\subset N$ is a disc transversal to $T$, tangent to $P$ at $x\in N$, and  area $\delta$.
	
	Gromov  \cite[page 73]{gromov}  defined ${\rm Area}_T(g/h_0)$ the {\em average area ratio of} $T$ to be  
	$${\rm Area}_T(g/h_0):=\frac{1}{{\rm vol}_{h_0}(Gr_2(N))}\int_{Gr_2(N)} |\wedge^2T^{-1}|(x,P) dV_{h_0}(x,P).$$
	A {more detailed}  definition is  in \eqref{av.area.ratio}.

	In the same paper, Gromov used a second variation argument to show that ${\rm Area}_T(g/h_0)\geq {\rm degree}\,T/3$ if the scalar curvature of $g$ satisfies $R(g)\geq -6$. The reason he got the factor $1/3$ is because the second variation of area is not sharp on hyperbolic manifolds. He conjectured in Remark 2.4.C' of \cite{gromov} that  assuming $R(g)\geq -6$
	\begin{center}
		{\em  ${\rm Area}_T(g/h_0)\geq {\rm degree}\,T$ holds true and it is sharp.}
	\end{center}
	This question was also addressed in Section 3.H  of \cite{gromov2} and in the more recent \cite{gromov4} at the end of Section 5.9.
	
	 For surfaces, the correspondent problem was solved by Katok. On a closed surface with two negatively curved metrics $\sigma, \sigma_1$,  Katok \cite{katok} defined the average length ratio $L(\sigma/\sigma_1):=\int_{T^1S}||v||_\sigma dV_{\sigma_1}(v)$ (he used a different notation). With $h(\sigma)$ denoting the topological entropy of $\sigma$,  he showed \cite[Theorem A]{katok} that $h(\sigma)L(\sigma/\sigma_1)\geq h(\sigma_1)$. Thus when $\sigma_1$ is hyperbolic and $\sigma$ has Gaussian curvature $K(\sigma)\geq -1$ we have $h(\sigma)\leq h(\sigma_1)$ and so $L(\sigma/\sigma_1)\geq 1$. 
	
	\subsection{Motivation} We follow the spirit of \cite{gromov,gromov2} and explain, informally, one of the  motivations behind ${\rm Area}_{\rm Id}(g/g_0)$. An important object in the study of negatively curved manifolds is the $1$-dimensional foliation of the unit tangent bundle where each leaf is {an orbit for the}  geodesic flow. 
	
	Gromov proposed the following surface analogue: A $2$-dimensional foliation $\mathcal L$ of $Gr_2(M)$ is a family $\{S_{\lambda}\}_{\lambda\in\Lambda}$ of complete surfaces of $M$ so that through every point in $(p,\tau)\in Gr_2(M)$ there is a unique $\lambda$ so that $S_\lambda$ passes through $p$ with tangent plane $\tau$. Laminations where each leaf is a stable minimal surface are analogous to the geodesic flow foliation of the unit tangent bundle. There is a canonical foliation $\mathcal L_0$ of $Gr_2(M)$ whose leafs are totally geodesic planes. 
	
	Given a metric $g$ on $M$ we obtain an area form on each leaf of $\mathcal L$.  Assuming   $\mathcal L$ has a transversal measure  we can integrate the area of each leaf with respect to this measure and obtain $vol_g(\mathcal L)$. 
	
	The general question Gromov asks is how low can {one} make $vol_g(\mathcal L)$ subject to the constraint that $R(g)\geq -6$. There is a  ``natural''  transversal measure on  $\mathcal L_0$ such that  ${\rm Area}_{{\rm Id}}(g/g_0)$ coincides with $vol_g(\mathcal L_0)/vol_{g_0}(\mathcal L_0)$. Therefore, Gromov's conjecture when $T={\rm Id}$ can be rephrased as $$vol_g(\mathcal L_0)\geq vol_{g_0}(\mathcal L_0)\quad\mbox{if}\quad R(g)\geq -6.$$
	
	The following comparison is also interesting: if we consider a $3$-dimensional foliation of $Gr_3(M)\simeq M$, there is only one leaf, the (unit) transverse measure is trivial, and given a metric on $g$, the volume of this foliation is  $vol_g(M)$. Schoen conjectured that  $vol_g(M)\geq vol_{g_0}(M)$ among all metrics with $R(g)\geq -6$ and this was proven by Perelman using Ricci flow (see \cite{anderson} for the argument).

	\subsection{Main Theorems}
	The next theorem uses Ricci flow in the spirit of  \cite{bben, marques-neves} and $\text {PSL}(2,\R)$-invariant laminar measures.
	\begin{thm}\label{mse.thm} Suppose $(M,g)$ has $R(g)\geq -6$. Then $${E_{\bar\mu}(g)\leq 2}$$ and equality occurs if and only if $g$ is isometric to the hyperbolic metric.
	\end{thm}
	
	This result was shown by the first author in \cite{lowe} under the condition that the (normalized) Ricci flow starting at $g$ converges to the hyperbolic metric faster than $e^{-ct}$ for some $c>1$. As it is pointed out in \cite{lowe}, there are  closed manifolds  $\H^3/ \Gamma$ for which that will happen for all metrics and  closed manifolds  $\H^3/ \Gamma$ with metrics for which that will not happen.  Manifolds containing closed embedded totally geodesic  surfaces fall into the second category. 
	Later the second author shared his notes containing an approach to the general case relying on $\text {PSL}(2,\R)$-invariant laminar measures. Both authors decided to join their efforts and write a single paper containing that and other related results.  
	
	{Next we outline a construction showing the existence of  a closed hyperbolic $3$-manifold $M$ having a metric $g$  with $R(g)\geq -6$ and  $E(g)>2$.  Necessarily $M$ must contain a closed totally geodesic surface. 
	
Consider $(N,g_0)$ a compact hyperbolic $3$-manifold whose boundary $\partial N$ is a totally geodesic closed surface and let $M$ be the doubling of $N$. The doubling of $g_0$  (also denoted $g_0$) is hyperbolic and thus $(M,g_0)$ is a closed hyperbolic $3$-manifold .

 Consider $L=[0,
+\infty)\times \partial N$ with the hyperbolic metric $dr^2+\cosh^2 rg_{\partial N}$  and glue it to $N$ to obtain a complete hyperbolic $3$-manifold $(F,g_1)$ with infinite volume and one Fuchsian end. In Section 5 of \cite{thurston} we see that we can deform the hyperbolic metric on $F$ to obtain a new complete hyperbolic metric $g_2$ on $F$. By minimizing area for $g_2$ in the homotopy class of $\partial N\subset F$ we find a compact region $N_F\subset F$ whose boundary is a minimal embedded surface homotopic to $\partial N$. If the boundary of  $N_F$ were totally geodesic, then $N_F$ would be isometric to $N$ and thus $g_1$ isometric to $g_2$. Hence ${\rm area}_{g_2}(\partial N_F)<{\rm area}_{g_0}(\partial N)$.

The doubling of $N_F$  is a smooth closed $3$-manifold diffeomorphic to $M$. Hence  the doubling of $g_2$ gives a piecewise smooth continuous metric on  $M$ which is hyperbolic where it is smooth. Using Theorem 4.2 of \cite{agol-storm-thurston} we find a metric $g_{\delta}$ on $M$ with scalar curvature $R(g_\delta)\geq -6$ and arbitrarily close in $C^0$ to $g_2$. If $\Pi$ denotes the homotopy class induced by the totally geodesic surface $\partial N\subset M$  we have $\partial N_F\in \Pi$ and  
	so $$\lim_{\delta\to 0}{\rm area}_{g_\delta}(\Pi)\leq {\rm area}_{g_2}(\partial N_F)<{\rm area}_{g_0}(\partial N)={\rm area}_{g_0}(\Pi).$$
From Proposition 2.1	 in \cite{marques-neves-conf} we have that $E(g_\delta)\geq 2  {\rm area}_{g_0}(\Pi)/ {\rm area}_{g_2}(\Pi)$ and so $E(g_\delta)>2$ for all $
\delta$ very small.}

	We  note that Agol, Storm, and Thurston \cite{agol-storm-thurston} conjectured that $E_{vol}(g)$ is maximized by the hyperbolic metric among all metrics with $R(g)\geq -6$.

	The next theorem relates the average area ratio and the minimal surface entropy in a sharp way.
	
	\begin{thm}\label{rel.thm} For every Riemannian metric  $(M,g)$  we have $${\rm Area}_{{\rm Id}}(g/g_0) {E_{\bar\mu}(g)}\geq 2$$ and  equality holds if and only if $g=cg_0$ for some constant $c>0$.
	\end{thm}
	Appealing to the previous interpretation of ${\rm Area}_{{\rm Id}}(g/g_0)$, we can restate Theorem \ref{rel.thm} as  saying that for every metric $g$
	$$vol_{g}(\mathcal L_0) E_{\bar\mu}(g)\geq 2 vol_{g_0}(\mathcal L_0) $$
	with equality if and only if  $g=cg_0$ for some $c>0$.
	This is reminiscent of Besson, Courtois, and Gallot \cite{BCG} which says that 
	$$vol_g(M)^{1/3}E_{vol}(g)\geq 2vol_{g_0}(M)^{1/3}$$
	for every metric $g$ with equality if and only if $g$ is isometric to $g_0$.   

	Combining Theorem \ref{mse.thm} and Theorem \ref{rel.thm} we have the following corollary.
	\begin{cor}\label{aar.thm} Consider $T:M\rightarrow N$ a local diffeomorphism with degree $d$. If $(M,g)$ has scalar curvature $R(g)\geq -6$ then
		$${\rm Area}_T(g/h_0)\geq d$$
		{and equality holds if and only if $T$ is a local isometry between $g$ and $h_0$.}
	\end{cor}
	This result confirms Gromov's conjecture for local diffeomorphisms. If we just assume that $T$ has degree $d$ there may exist no local isometry between $M$ and $N$ and thus no obvious optimal map. 
	\begin{proof}[Proof of Corollary \ref{aar.thm}]
		We have 
		$${\rm Area}_{T}(g/h_0)=d\,{\rm Area}_{{\rm Id}}(g/T^*(h_0)).$$
		Hence, with $g_0:=T^*(h_0)$, we have from Theorem \ref{rel.thm} and Theorem \ref{mse.thm} that
		$$2{\rm Area}_{T}(g/h_0)\geq {E_{\bar\mu}(g)}{\rm Area}_{T}(g/h_0)=d {E_{\bar\mu}(g)}{\rm Area}_{{\rm Id}}(g/g_0)\geq 2d$$
		and equality holds if and only if $g=g_0=T^*(h_0)$.
	\end{proof}
	
	
	\subsection{Sketch of proofs} We describe succinctly the main ideas behind the proofs of Theorem \ref{mse.thm} and Theorem \ref{rel.thm}. 
	
	Assume that $(M,g)$ has $R(g)\geq -6$.  To prove Theorem \ref{mse.thm} we need to show that for every sequence of homotopy classes $\Pi_m\in S_{1/m,\bar \mu}(M)$ we have
	$$\lim_{m\to\infty}{{\rm area}_g(\Pi_m)}/{{\rm area}_{g_0}(\Pi_m)}\geq 1.$$
	Let $(g_t)_{t\geq 0}$ be a solution to normalized Ricci flow, which we first assume exists for all time and converges to $g_0$.  We show that if the inequality above fails then for all $m$ sufficiently large we have  for some $\delta>0$ and all $t\geq 0$
	$$ {{\rm area}_{g_t}(\Pi_m)}/{{\rm area}_{g_0}(\Pi_m)}\leq 1-\delta e^{-t}.$$
	Stability analysis shows that $g_t\simeq g_0+e^{-t}\bar h$ as $t\to\infty$, where $\bar h$ is an eigentensor {for the linearization of the trace-free Ricci tensor at the hyperbolic metric}. Letting $m\to\infty$ we obtain 
	{measures} $\mu_t$ on the frame bundle of $M$ so that $$\mu_t(1)=\lim_{m\to\infty } {{\rm area}_{g_t}(\Pi_m)}/{{\rm area}_{g_0}(\Pi_m)}.$$ Necessarily $\mu_t(1)\leq 1-\delta e^{-t}$. Using a form of Gauss equation (see \eqref{identity})  for these measures {we have an identity of the type
		$$1=\mu_t(1)+\mu_t(\mbox{curvature terms}).$$
		Combining with the  asymptotics $g_t\simeq g_0+e^{-t}\bar h$} we show
	$$1=\mu_t(1)+\mu_t(\mbox{curvature terms})=\mu_t(1)+e^{-t}\mu_{+\infty}(\mbox{terms with }\bar h)+o(e^{-t}).$$
	{The measure $\mu_{+\infty}$ {coincides with the homogeneous Liouville measure $\bar \mu$ and from the fact that $\bar h$ is a trace free eigentensor,}  we show that  $$\mu_{+\infty}(\mbox{terms with }\bar h)=0.$$ Thus $1= \mu_t(1)+o(e^{-t})$, which contradicts $\mu_t(1)\leq 1-\delta e^{-t}$.
	
	For the general  case, we proceed as above but use Perelman's Ricci flow with surgery \cite{perelman}. Because $M$ has an hyperbolic metric all surgeries correspond to removing capped horns and replacing them by standard caps. The key point to check is that essential  surfaces minimizing area stay away from the capped horns. We achieve this via area comparison.
	
	We now sketch the argument behind Theorem \ref{rel.thm}. We start by improving a construction of Labourie in \cite{Labourie} and find  a sequence $\Sigma_m\subset M$ of connected immersed minimal surfaces with respect to $g_0$ which becomes equidistributed in the frame bundle.
	We show in Proposition \ref{compute}  that $${\rm Area}_{{\rm Id}}(g/g_0)=\lim_{m\to\infty}{\rm area}_g(\Sigma_m)/{\rm area}_{g_0}(\Sigma_m).$$
	A counting argument implies that $$E_{\bar\mu}(g)\geq E(g_0)\lim_{m\to\infty}{\rm area}_{g_0}(\Sigma_m)/{\rm area}_{g}(\Sigma_m)$$
	and these two expressions give that $E_{\bar\mu}(g){\rm Area}_{{\rm Id}}(g/g_0)\geq 2$. If equality holds we show first that the metric $g$ is Zoll, i.e., every totally geodesic disc with respect to $g_0$ is minimal with respect to $g$, and then we show that  $g$ must be a multiple of $g_0$. This proves Theorem \ref{rel.thm}.

	\subsection{Acknowledgments} 
	
	The first author thanks his advisor Fernando Cod\'a Marques for helpful conversations related to this paper and for his support. 
	
	The second author thanks Danny Calegari. We are also thankful to the referee for the suggestions which helped improved the presentation.

	\section{Preliminaries}\label{prelim}

	Assume $(M,g)$ is a closed Riemannian $3$-manifold admitting a hyperbolic metric $g_0$ and $N$  another closed hyperbolic $3$-manifold  $N$ with an hyperbolic metric $h_0$.

	Suppose $T:M\rightarrow N$ is  smooth map. {Given $(x,P)\in Gr_2(M)$ we denote by $|\wedge^2T|_g(x,P)$ the Jacobian of  $$dT_x:P\rightarrow T_{T(x)}{N},$$ meaning that if $e_1,e_2$ is an orthonormal basis of $P$  and $u_i:=dT_x(e_i),$ $i=1,2,$
		$$|\wedge^2T|_g(x,P):=\sqrt{h_0(u_1,u_1)h_0(u_2,u_2)-h_0^2(u_1,u_2)}.$$}
	Given $y\in N$ a regular value {and $\tau=(y,V)\in  Gr_2(N)$ set
		$$|\wedge^2T^{-1}|_g(y,V)=\sum_{x\in T^{-1}(y)}\frac{1}{|\wedge^2T|_g(x,(dT_x)^{-1}(V))}$$
		The function $\tau\mapsto|\wedge^2T^{-1}|_g(\tau)$} is defined almost everywhere on $Gr_2(N)$.
	
	For context, if $T$ is transverse to a closed surface $S\subset {N}$, then $\Sigma:=T^{-1}(S)$ is a surface of ${M}$  
	{with area} 
	\begin{equation*}\label{coarea}
		{\rm area}_g(\Sigma)=\int_{{S}}|\wedge^2T^{-1}|_g(y,T_yS)dA_{h_0}(y).
	\end{equation*}
	
	Gromov \cite[page 73]{gromov}  defined the {\em average area ratio of} $T$ as
	\begin{equation}\label{av.area.ratio}
		{\rm Area}_T(g/h_0):=\int_{Gr_2(N)}|\wedge^2T^{-1}|_g(\tau)d\mu_{h_0}(\tau),
	\end{equation}
	where integration is with respect to the {\em unit} volume measure $\mu_{h_0}$  on $Gr_2(N)$ induced by $h_0$. In particular, ${\rm Area}_{{\rm Id}}(h_0/h_0)=1$. {In the definition above it is implicitly assumed that $|\wedge^2T^{-1}|_g$ is in $L^1$. If that is not the case we define ${\rm Area}_T(g/h_0)=\infty$}.
	
	\medskip

	We use  $\pi_1(M)$ to denote as well its representation into $\text{PSL}(2,\C)$.\footnote{Using the Poincar\'e ball model, the orientation preserving isometries of $\H^3$ are identified with the automorphisms of $S^2$ which is the group $\text{PSL}(2,\C).$} A closed immersed surface $\Sigma\subset M$ is {\em essential} if the immersion $\iota:\Sigma \rightarrow M$ is $\pi_1$-injective. Essential surfaces lift to discs in $\H^3$. Using the representation of $\pi_1(M)$ into ${\rm PSL}(2,\C)$ we have $\iota_{*}(\pi_1(\Sigma))< \text{PSL}(2,\C)$ and so we can associate its limit set $\Lambda(\Sigma)\subset \partial_{\infty}\H^3\simeq S^2$ (for the definition see for instance \cite[Section 2.2]{cal-marques-neves}).  

	Set  $\mathcal C_\varepsilon$ to be  the space of $(1+\varepsilon)$-quasicircles in $\partial_{\infty}\H^3$ (see \cite[Definition 2.3]{cal-marques-neves} for precise definition). The group $\pi_1(M)$ acts on  $\mathcal C_\varepsilon$ and preserves $\mathcal C_0$ (which is the space of all geodesic circles in $S^2\simeq\partial_{\infty}\H^3$). 
	
	Let $S_{\varepsilon}(M)$ denote the set of all homotopy class $\Pi$ of essential surfaces with limit set in $\mathcal C_\varepsilon$. Recall that we defined
	$$\text{area}_{g}(\Pi):= {\text{inf}} \{\text{area}_{g}(S): S\in \Pi\}.$$
	From Schoen-Yau  \cite{schoen-yau} there is an immersed minimal surface (with respect to $g$) ${\bf \Sigma}_g(\Pi)\in\Pi$ which realizes $\text{area}_{g}(\Pi)$.

	From \cite{seppi}, assuming $\varepsilon$ is sufficiently small, for all $\gamma\in \mathcal C_\varepsilon$ there is a unique embedded area-minimizing   disc $D(\gamma)\subset \H^3$  (with respect to $g_0$) with $\partial_{\infty}D(\gamma)=\gamma\subset \partial_{\infty}\H^3$ and principal curvatures that can be made arbitrarily small by choosing $\varepsilon$ small enough. 
	
	The same argument as in Theorem 4.3 of \cite{lowe2} (adapted to the case {where the minimal discs}  are not necessarily preserved by some surface group),  shows that we can find a neighborhood $\mathcal U$ of $g_0$ and $\bar \varepsilon$ small so that for each $\gamma\in \mathcal C_{\bar \varepsilon}$ there is a unique non-degenerate area-minimizing disc $\Sigma_g(\gamma)$ with respect to the metric $g$ so that $\partial_{\infty}\Sigma_{g}(\gamma)=\gamma$. The discs $\Sigma_g(\gamma)$ and $D(\gamma)$ are at a bounded Hausdorff distance from each other (independent of $g$) and if $g\to g_0$ then $\Sigma_g(\gamma)$ converges to $D(\gamma)$ uniformly in $C^{2,\alpha}$.  {Therefore  we} can decrease $\mathcal U$ and $\bar \varepsilon$ so that for all $\gamma\in  \mathcal C_{\bar \varepsilon}$ there is $f_\gamma\in C^{\infty}(D(\gamma))$ (depending on $g$) such that its graph over $D(\gamma)$ is identical to $\Sigma_g(\gamma)$ and {$|f_{\gamma}|_{2,\alpha}<1$.} There is an ambiguity on the sign of $f_\gamma$ but the quantities we consider, like $|f_\gamma|$, will be sign independent. 
	
	{If $\Pi\in S_{\bar \varepsilon}(M)$ and $g\in \mathcal U$ then, with  {$\gamma:=\Lambda({\bf \Sigma}_g(\Pi))$},  uniqueness implies that $\Sigma_g(\gamma)$ covers a minimal surface ${\bf \Sigma}_g(\gamma)$ in $\Pi$ which must coincide with ${\bf \Sigma}_g(\Pi)$ and thus satisfy  $$\text{area}_g(\Pi)=\text{area}_g({\bf \Sigma}_g(\gamma)).$$}
	{Given $\gamma\in \mathcal C_{\varepsilon}$, let $n(\gamma)$ denote a  continuous unit normal vector field along $D(\gamma)$ with respect to $g_0$.} Consider  the diffeomorphism (using the hyperboloid model)
	\begin{equation}\label{mapF}
		F_{\gamma}: D(\gamma)\rightarrow \Sigma_{g}(\gamma),\quad x\mapsto \cosh(f_{\gamma}(x))x+\sinh(f_{\gamma}(x))n(\gamma)(x).
	\end{equation}
	We omit the dependence of $g$ on $F_{\gamma}$ to avoid too much notation. When $g=g_0$,  $f_{\gamma}=0$ and thus $F_{\gamma}$ is the identity.
	
	\section{Laminations and Laminar measures}\label{measures.labourie}
	
	We follow the presentation of Labourie in \cite{Labourie} and add some auxiliary results.
	\subsection{Laminations}
	Consider the space of {\em stable minimal conformal immersions} $\mathcal F(\H^3,\varepsilon)$ (with $\varepsilon\leq \bar\varepsilon$) defined in Definition 5.1 of \cite{Labourie}, i.e.,  the space of conformal minimal immersions $$\phi:\H^2\rightarrow (\H^3,g_0)\quad{\rm with}\,\, \partial\phi:=\phi(\partial_{\infty}\H^2)\in \mathcal C_{\varepsilon}.$$ Because $\varepsilon\leq \bar \varepsilon$, $D(\partial\phi)=\phi(\H^2)$, $\phi$ is an embedding and stable for the second variation of area. The topology we choose is the same as the one considered in Definition 5.1 of \cite{Labourie} and it makes the map $\phi\mapsto \partial\phi$ continuous (Theorem 5.2 of \cite{Labourie}). Thus $$\cap_{k\in\N}\mathcal F(\H^3,1/k)=\mathcal F(\H^3,0).$$

	Similar to \cite{Labourie}  we also consider $\mathcal F(M,\varepsilon):=\mathcal F(\H^3,\varepsilon)/ \pi_1(M)$ {with the quotient topology}. The group $\text{PSL}(2,\R)$ acts on $\H^2$ and thus it acts on $\mathcal F(M,\bar\varepsilon)$  in the following way:   
	\begin{equation}\label{action}
		\tau\in  \text{PSL}(2,\R), \quad R_\tau: \mathcal F(M,\varepsilon)\rightarrow  \mathcal F(M,\varepsilon)\quad  R_\tau(\phi):=\phi\circ \tau^{-1}.
	\end{equation}
	The space  $\mathcal F(M,\bar\varepsilon)$ together with the $\text{PSL}(2,\R)$-action is called the {\em conformal minimal laminations} of $M$.

	Fix a  fundamental domain $\Delta\subset \H^3$ of $M$. Given $\phi\in \mathcal F(M,\varepsilon)$, there is a unique lift to an element of $\mathcal F(\H^3,\varepsilon)$, denoted by $\phi$ as well, that is uniquely determined by the requirement that $\phi(i)\in \Delta$.   Thus for each $\phi \in \mathcal F(M,\varepsilon)$ we obtain $\partial \phi\in \mathcal C_{\varepsilon}$ but this map is not necessarily continuous. Theorem 5.2 (i) of \cite{Labourie} says that the evaluation map which sends $\phi\in \mathcal F(\H^3,\bar \varepsilon)$ to $\phi(i)\in \H^3$ is proper. As a result we deduce at once the lemma below.
	
	\begin{lemm}\label{laminations.compact} The space $\mathcal F(M,\bar \varepsilon)$ is sequentially compact.
	\end{lemm}

	Given $\phi\in \mathcal F(M,\bar\varepsilon)$, let $C(\phi)> 0$ be the conformal factor of $\phi^*(g_0)$. Denote the Gaussian curvature of $D(\partial \phi)$ by $K(\phi)$. From Gauss equation {and \cite{seppi} we have, after making $\bar \varepsilon$ smaller if necessary, $-2\leq K(\phi)\leq -1$.} The maximum principle  applied to the equation satisfied by $C(\phi)$ implies that
	\begin{equation}\label{gauss.phi}
		\frac{1}{2}\leq \frac{1}{\sup_{D(\partial\phi)} |K(\phi)|}\leq C(\phi)\leq 1.
	\end{equation}

	Let $F(M)$ denote the frame bundle of $M$, i.e., $F(M)= \text{PSL}(2,\C)/ \pi_1(M)$. Fix $\{e_1,e_2\}$ an orthonormal basis of ${T_i}\H^2$ and, given $\phi\in  \mathcal F(M,\bar\varepsilon)$, set
	$$e_1(\phi):=d\phi_i(e_1)C(\phi)^{-1/2}\quad\text{and}\quad e_2(\phi):=d\phi_i(e_2)C(\phi)^{-1/2}.$$
	There  is a unique $n(\phi)\in T_{\phi(i)}M$ so that $\{e_1(\phi),e_2(\phi),n(\phi)\}$ is a positive frame.
	Consider the {continuous} map
	\begin{equation}\label{omega.map}
		\Omega: \mathcal F(M,\bar\varepsilon)\rightarrow F(M),\quad \Omega(\phi)=(\phi(i),\{e_1(\phi),e_2(\phi),n(\phi)\}).
	\end{equation}

	Given $\phi\in \mathcal F(M,\bar\varepsilon)$,  set $\gamma:=\partial \phi$ and denote the Jacobian of  (see \eqref{mapF})
	$$F_{\gamma}\circ \phi:\H^2\rightarrow (\Sigma_{g}(\gamma),g)$$ by $|Jac_g(F_{\gamma}\circ \phi)|$. Consider the  function
	\begin{equation}\label{jac.fcn}
		\Lambda_g: \mathcal F(M,\bar\varepsilon)\rightarrow \R, \quad \Lambda_g(\phi):=|Jac_g(F_{\gamma}\circ \phi)|(i).
	\end{equation}
	If $dA_g$ is the area element of $\Sigma_{g}(\gamma)$, then $(F_{\gamma}\circ\phi)_i^*(dA_g)= \Lambda_g(\phi)dx\wedge dy$, in isothermal coordinates. This function is continuous because it is independent of the particular lift of $\phi$ that was chosen.

	With $\gamma\in\mathcal C_{\bar \varepsilon}$,  let  $\nu_g(\gamma)$ and  $|A|_g^2(\Sigma_{g}(\gamma))$ denote respectively a  continuous unit normal vector field along $\Sigma_g(\gamma)$ with respect to $g$ and the norm square of the second fundamental form of $\Sigma_{g}(\gamma)$ with respect to the metric $g$. 
	Consider the following functions 
	\begin{equation}\label{fA}
		|A|_g^2:\mathcal F(M,\varepsilon)\rightarrow \R, \quad \phi\mapsto |A|_g^2(\Sigma_{g}(\partial\phi))(F_{\partial \phi}\circ\phi(i)),
	\end{equation}
	\begin{equation}\label{fRic}
		Ric(g)(\nu,\nu):\mathcal F(M,\varepsilon)\rightarrow \R, \quad \phi\mapsto Ric(g)_{|F_{\partial \phi}\circ\phi(i)}(\nu_g(\partial\phi),\nu_g(\partial\phi)),
	\end{equation}
	\begin{equation}\label{fR}
		R(g):\mathcal F(M,\varepsilon)\rightarrow \R, \quad \phi\mapsto R(g)(F_{\partial \phi}\circ\phi(i)).
	\end{equation}
	The definition of all these functions is independent of the particular lift of $\phi\in\mathcal F(M,\varepsilon)$ that was chosen and thus they are  continuous.

	\subsection{Laminar measures} A {\em laminar measure} $\mu$ on $\mathcal F(M,\bar\varepsilon)$ is a probability measure that is invariant under the $\text{PSL}(2,\R)$-action given by \eqref{action}.

	A laminar measure $\mu$ and the map $\Omega$ defined in \eqref{omega.map} induces a probability measure $\Omega_*\mu$ on $F(M)$. That measure is invariant under a $\text{PSL}(2,\R)$-action which will not coincide in general with the homogeneous action of $\text{PSL}(2,\R)$ as a subgroup of $\text{PSL}(2,\C)$. Another issue that needs to be addressed is the fact that the space of laminar measures is not necessarily weakly compact (a related problem is put as an open question in \cite{Labourie2}).

	\begin{lemm}\label{compactness.lemm} Let $\mu_k$ be a sequence of   laminar measures on $\mathcal F(M,1/k)$ so that $\Omega_*\mu_k$ converges weakly to a probability measure $\bar\mu$ on $F(M)$.  Then $\bar \mu$  is  invariant under  the homogeneous action of $\text{PSL}(2,\R)$. 
	\end{lemm}
	\begin{proof}
		Let $\Omega_0:\mathcal F(M,0)\rightarrow F(M)$ be the restriction of $\Omega$ to $\mathcal F(M,0)$. Every  $\phi \in \mathcal F(M,0)$ has the property that $D(\partial\phi)$ is a totally geodesic disc and so $\phi:\H^2\rightarrow M$ is an isometric immersion. Thus $\phi$ is uniquely determined by $\phi(i), e_1(\phi),$ and $e_2(\phi)$. Hence $\Omega_0$ is bijective and a homeomorphism.
		
		Recall the $ \text{PSL}(2,\R)$-action on  $\mathcal F(M,\bar\varepsilon)$ defined in \eqref{action}. $\Omega_0$ induces a $\text{PSL}(2,\R)$-action on $F(M)$ in the following way: 
		$$\tau\in \text{PSL}(2,\R), \quad \bar R_\tau: F(M)\rightarrow F(M), \quad \bar R_\tau(x)= \Omega_0(R_{\tau}\circ \Omega_0^{-1}(x)).$$
		This action corresponds to the homogeneous action of $\text{PSL}(2,\R)$. Thus, given $f\in C^0(F(M))$, we need to check that  $$\bar \mu(f\circ \bar R_\tau)=\bar \mu(f)\quad\mbox{for all }\tau\in\text{PSL}(2,\R).$$
		Consider the ``projection'' of $\mathcal F(M,\varepsilon)$ onto $\mathcal F(M,0)$ given by $$P:=\Omega_0^{-1}\circ\Omega:\mathcal F(M,\varepsilon)\rightarrow \mathcal F(M,0).$$  
		Set $\eta:=f\circ \Omega\circ R_\tau$. We have $f\circ \bar R_\tau\circ\Omega=\eta\circ P$ and thus
		\begin{equation*}
			{\bar\mu}(f\circ \bar R_\tau)=\lim_{k\to\infty}\Omega_*\mu_k(f\circ \bar R_\tau)=\lim_{k\to\infty}\mu_k(\eta\circ P).
		\end{equation*}
		We also have $\Omega_*\mu_k(f)=\mu_k(\eta)$ for all $k\in\N$ by  $\text{PSL}(2,\R)$-invariance and thus
		\begin{equation*}
			{\bar\mu} (f)=\lim_{k\to\infty}\Omega_*\mu_k(f)=\lim_{k\to\infty}\mu_k(\eta).
		\end{equation*}
		In light of these last two identities it suffices to check that
		$$\lim_{k\to \infty}|\mu_k(\eta\circ P)-\mu_k(\eta)|=0$$
		{and this follows at once if we show that}
		$$\lim_{k\to \infty}\sup_{\phi\in  \mathcal F(M,1/k)}|\eta\circ P(\phi)-\eta(\phi)|=0.$$
		{If this identity does not hold we find $\delta>0$ and $\phi_k\in  \mathcal F(M,1/k)$ so that $|\eta\circ P(\phi_k)-\eta(\phi_k)|\geq \delta$ for all $k\in\N$.
			From Lemma \ref{laminations.compact} we know that, after passing to a subsequence, $\phi_k$ converges to some $\phi\in  \mathcal F(M,0)$ which must satisfy $|\eta\circ P(\phi)-\eta(\phi)|\geq \delta$. This is impossible because $P(\phi)=\phi$.}
	\end{proof}

	Let $\Gamma$ be a Fuchsian subgroup of $\text{PSL}(2,\R)$  so that $\H^2/ \Gamma$ is a  closed hyperbolic surface with genus $l$. All Fuchsian groups we consider will have this property with no need for further mentioning.
	
	$\text{PSL}(2,\R)/ \Gamma$ is the frame bundle of $\H^2/ \Gamma$. With respect to the invariant metric on $\text{PSL}(2,\R)/ \Gamma$ we have $vol(\text{PSL}(2,\R)/ \Gamma)=\alpha_04\pi(l-1)$ for some universal constant $\alpha_0$.
	
	Suppose $\phi \in\mathcal F(M,\bar\varepsilon)$ is equivariant with respect to a representation of $\Gamma< \text{PSL}(2,\R)$ in $\pi_1(M)<\text{PSL}(2,\C)$. Consider $U\subset \text{PSL}(2,\R)$ a fundamental domain of $\text{PSL}(2,\R)/ \Gamma$. Following Proposition 5.5 of \cite{Labourie} we define $\delta_{\phi}$ a laminar measure on $\mathcal F(M,\bar\varepsilon)$
	$$\delta_\phi(f):=\frac{1}{vol(U)}\int_U f(\phi\circ \tau)d\bar\nu(\tau),\quad f\in C^0(\mathcal F(M,\bar\varepsilon))$$
	where $\bar\nu$ is the bi-invariant measure on $\text{PSL}(2,\R)$.
	
	Consider the covering map $\pi:\H^3\rightarrow M$. Equivariance implies that $D(\partial \phi)$ projects to a closed surface  ${\bf D(\partial \phi)}$ in $M$ and that $\phi$ descends to an immersion from $\H^2/ \Gamma$ to ${\bf D(\partial \phi)}$ that we still denote by ${\phi}$. The uniqueness property implies that $\Sigma_{g}(\partial \phi)$ projects to a closed surface ${\bf \Sigma}_{g}(\partial \phi)$ on $M$ that is homotopic to ${\bf D(\partial \phi)}$. The map  $F_{\partial \phi}\circ \phi $ is also equivariant and so descends to a map from 
	$\H^2/ \Gamma$ to ${\bf \Sigma}_{{g}}(\partial \phi)$ that we denote by $\bf{F_{\partial \phi}}\circ {\phi}$ \footnote{The maps $\bf{F_{\partial \phi}}\circ {\phi}$ and ${\pi\circ F_{\partial \phi}}\circ\phi$ have the same image but different domains.}.
	The unit normal vector field $\nu_g(\gamma)$ induces a unit normal vector field along  ${\bf \Sigma}_g(\gamma)$ that we also denote by $\nu_g(\gamma)$.
	
	{Given $\Pi\in S_{\bar \varepsilon}(M)$ there is $\phi\in \mathcal F(M,\bar\varepsilon)$ equivariant with respect to some Fuchsian subgroup of $\text{PSL}(2,\R)$ so that ${\bf D(\partial \phi)}\in \Pi$. The laminar measure $\delta_\phi$ depends only on $\Pi$ and so we denote it by $\delta_\Pi$. Indeed if $\phi'\in \mathcal F(M,\bar\varepsilon)$ is such that ${\bf D(\partial \phi')}\in \Pi$ we have ${\bf D(\partial \phi')}={\bf D(\partial \phi)}$ and so $\phi'\circ \phi^{-1}\in\text{PSL}(2,\R)$, which implies that $\delta_\phi=\delta_{\phi'}$. 
	
	Given $\Pi\in S_{\bar \varepsilon}(M)$ we consider the  unit measure on $F(M)$ given as 
	\begin{equation}\label{pi.measure}
	\mu_{\Pi}=\Omega_*\delta_\phi.
	\end{equation}
	}

	For context, suppose $f$ is a continuous function in $Gr_2(M)$ and set
	$$\hat f:\mathcal F(M,\bar\varepsilon)\rightarrow \R,\quad \phi\mapsto f({\pi\circ F_{\partial \phi}}\circ\phi(i),d({\pi\circ F_{\partial \phi}})_{|\phi(i)}(d\phi_{|i}({T_i}\H^2))).$$
	The function $\hat f$ is continuous. With $dA_{hyp}$ denoting the hyperbolic volume form on $\H^2$ we have
	\begin{multline}
		\int_{{\bf \Sigma}_{g}(\partial \phi)}f(x,T_x{\bf \Sigma}_{g}(\partial \phi))dA_g(x)\\
		=\int_{\H^2/ \Gamma} f({\bf F_{\partial \phi}}\circ\phi(z),d{\bf F_{\partial \phi}}_{|\phi(z)}(d_z\phi({T_z}\H^2))|Jac_g ({\bf F_{\partial \phi}\circ \phi})|(z)dA_{hyp}(z)\\
		=\frac{1}{\alpha_0}\int_U\hat f(\phi\circ \tau)\Lambda_g(\phi\circ \tau) d\bar \nu(\tau)=4\pi(l-1)\delta_{\phi}(\hat f \Lambda_g).
	\end{multline}
	
	\subsection*{Gauss identity for laminar measures}
	From Gauss equation we have 
	\begin{multline*}
		4\pi(l-1)=\text{area}_g({\bf \Sigma}_{g}(\partial \phi))\\
		+\int_{{\bf \Sigma}_{g}(\partial \phi)}Ric(g)({\nu}_g({\partial\phi}),{\nu}_g({\partial\phi})) -\frac{1}{3}R(g)+\frac{|A|^2}{2}-\frac{R(g)+6}{6}dA_g.
	\end{multline*}
	{When interpreted in terms of laminar measures this identity becomes}
	\begin{multline}\label{identity}
		1=\frac{\text{area}_g({\bf \Sigma}_{g}(\partial \phi))}{4\pi(l-1)}+\delta_{\phi}\left(\left[Ric(g)(\nu,\nu)-{R(g)/3}+{|A|^2_g}/{2}\right]\Lambda_g\right)\\
		-\delta_{\phi}\left(\frac{R(g)+6}{6}\Lambda_g\right),
	\end{multline}
	where the functions $|A|^2_g$, $Ric(g)(\nu,\nu)$, and $R(g)$ are as defined in \eqref{fA}, \eqref{fRic}, and \eqref{fR}.

	\section{Proof of Theorem \ref{mse.thm}: Part I}

	Recall that $M$ is a closed manifold with an hyperbolic metric $g_0$. Throughout  this paper we refer to normalized Ricci flow as a  one-parameter family of metrics $(\bar g_t)_{t\in I}$ which solve
	\begin{equation}\label{RF}
		\frac{d\bar g_t}{dt}=-2Ric(\bar g_t)-4\bar g_t.
	\end{equation}
	
	\begin{thm}\label{asympt.thm} There is a neighborhood $\mathcal V$ of $g_0$ so that for all $g\in \mathcal V$ with $R(g)\geq -6$ the following holds:
		For any sequence  $\Pi_m\in {S_{1/m,\bar \mu}(M)}$ we have
		$$\liminf_{m\to\infty}\frac{{\rm area}_g(\Pi_m)}{4\pi(l_m-1)}\geq 1,$$
		where $l_m$ is the genus of an essential surface in $\Pi_m$.
		
		If equality holds then $g$ is isometric to $g_0$.
	\end{thm}
	
	\begin{proof}

		Consider the neighborhood $\mathcal U$ of $g_0$ described in Section \ref{prelim}. From \cite[Appendix A]{knopf} we see that we can find  a neighborhood  of $g_0$ in the $C^{2,\alpha}$-topology so that for every initial condition  in that neighborhood, the normalized Ricci flow exists for all time and converges exponentially fast  in the $C^{2,\alpha}$-topology to an hyperbolic metric in $\mathcal U$. Reasoning like in \cite[Section 17]{hamilton} we can upgrade the convergence and find a  smaller open neighborhood $\mathcal V\subset {\mathcal U}$ of $g_0$  so that for every $g\in \mathcal V$ the normalized Ricci flow $(\bar g_t)_{t\geq 0}$
		starting at $g$ exists for all time, does not leave $\mathcal U$, and converges  exponentially fast to an Einstein metric in $\mathcal U$, which must be isometric to $g_0$ from Mostow rigidity.  Furthermore, in \cite{knopf} it is also constructed a family of diffeomorphisms $\{\Phi_t\}_{t\geq 0}$ converging strongly to some diffeomorphism so that  $g_t:=\Phi_t^*\bar g_t$ solves the DeTurck-modified Ricci flow (which is strictly parabolic) and $g_t$ converges to $g_0$ as $t\to\infty$.

		The maximum principle implies that the condition $R(g)\geq -6$ is preserved by the normalized Ricci flow because
		$$\frac{d}{dt}R(\bar g_t)\geq \Delta_{g_t}R(\bar g_t)+\frac{2}{3}R(\bar g_t)(R(\bar g_t)+6).$$

		\subsection{Proof of inequality} Suppose for contradiction that the inequality fails for some metric $g\in \mathcal V$.  Thus we can find  $\Pi_m\in S_{1/m,\bar \mu}(M)\cap S_{\bar\varepsilon}(M)$ and $\delta>0$  so that for all $m\in \N$
		$$\frac{\text{area}_g(\Pi_m)}{4\pi(l_m-1)} \leq 1-\delta.$$

		Because $\bar g_t\in \mathcal U$ for all $t$, there is $\phi_m\in \mathcal F(M,1/m)$ equivariant  with respect to a representation of a Fuchsian subgroup of  ${\rm PSL}(2,\R)$ in $\pi_1(M)$ so that ${\bf \Sigma}_{g_t}(\partial \phi_m)\in \Pi_m$, ${\bf \Sigma}_{g_t}(\partial \phi_m)$ depends smoothly on $t$, and $\text{area}_{g_t}(\Pi_m)=\text{area}_{g_t}({\bf \Sigma}_{g_t}(\partial \phi_m))$  for all $t\geq 0$. Fix $m\in\N$ and denote by $\nu_t$ the normal vector to ${\bf \Sigma}_{g_t}(\partial \phi_m)$ with respect to $g_t$. Using $R(g_t)\geq - 6$ and Gauss equation
		\begin{multline}\label{area.evolution}
			\frac{d}{dt}\text{area}_{g_t}(\Pi_m) = -\int_{{\bf \Sigma}_{g_t}(\partial \phi_m)}R(g_t)-Ric(g_t)(\nu_t,\nu_t)+4\,dA_{g_t}\\
			=4\pi(l_m-1)-\text{area}_{g_t}(\Pi_m)-\int_{{\bf \Sigma}_{g_t}(\partial \phi_m)} \frac{|A|^2}{2}+\frac{R(g_t)+6}{2}dA_{g_t} \\
			\leq 4\pi(l_m-1)-\text{area}_{g_t}(\Pi_m).
		\end{multline}
		Solving the ODE obtained by replacing the inequality sign above by an equality sign we get that for all $t\geq 0$
		\begin{equation}\label{loweRF}
			\frac{\text{area}_{g_t}(\Pi_m)}{4\pi(l_m-1)}\leq 1-e^{-t}\left(1-\frac{\text{area}_g(\Pi_m)}{4\pi(l_m-1)}\right)\leq 1-\delta e^{-t}.
		\end{equation}

		Combining \eqref{identity} with \eqref{loweRF} and using the fact that $R(g_t)+6\geq 0$ we get
		\begin{equation}\label{fund.identity}
			\delta e^{-t}\leq \delta_{\phi_m}\left(\left[Ric(g_t)(\nu,\nu)-{R(g_t)/3}+{|A|^2_{g_t}}/{2}\right]\Lambda_{g_t}\right).
		\end{equation}

		After passing to a subsequence, ${\Omega_*}\delta_{\phi_m}$ converges weakly to {the Liouville  measure $\bar \mu$ on $F(M)$ because $\Pi_m\in S_{1/m,\bar \mu}(M).$}

		Let $L$ denote the linearization of  the traceless Ricci tensor at $g_0$. For every $2$-tensor $h$ on $M$ set $\theta(h):F(M)\rightarrow \R$ to be the continuous  function given by
		\begin{equation}\label{theta.map}
			\theta(h)(x,\{u_1,u_2,n\}):=L(h)_{x}(n,n).
		\end{equation}
		
		Set $h_t:=g_t-g_0$. We use $O(\beta)$ to denote a term bounded by  $C\beta$, where $C$ does not depend on $m$ or $t$.
		\begin{lemm}
			$\delta \leq \bar\mu(e^t\theta(h_t))+O(e^{-1/3t})$.
		\end{lemm}
		\begin{proof}
			
			Write
			\begin{itemize}
				\item $E_1:=\delta_{\phi_m}({|A|^2_{g_t}}/{2}\,\Lambda_{g_t})$;
				\item $ E_2:=\delta_{\phi_m}([Ric(g_t)(\nu,\nu)-{R(g_t)/3}](\Lambda_{g_t}-1))$;
				\item$E_3:=\delta_{\phi_m}([Ric(g_t)(\nu,\nu)-{R(g_t)/3}]-\theta(h_t)\circ \Omega)$.
			\end{itemize}
			We now estimate these terms. We know from Proposition \ref{ricci-stable} that $$|h_t|_{C^4}=O(e^{-2/3t}).$$
			
			Fix $m\in\N$ and set $\gamma_m:=\partial \phi_m$. The surface $\Sigma_{g_t}(\gamma_m)$ has zero mean curvature with respect to $g_t$. Thus the mean curvature $H_{g_0}(\Sigma_{g_t}(\gamma_m))$ of $\Sigma_{g_t}(\gamma_m)$ with respect to the hyperbolic  metric   has $C^{0,\alpha}$-norm bounded  by $O(e^{-2/3t})$. (This follows from the fact that $h_t$ and $\nabla h_t$ have both that order).

			Using the mean convex foliation of $\H^3$ coming from the discs equidistant to $D(\gamma_m)$ that was described by Uhlenbeck in \cite[Theorem 3.3]{uhlenbeck}, the maximum principle implies that  for some constant $c_0$
			$$\sup_{m\in\N} |f_{\gamma_m}|_{L^\infty}\leq c_0\sup_{m\in\N} |H_{g_0}(\Sigma_{g_t}(\gamma_m))|_{L^{\infty}}=O(e^{-2/3t}).$$
			Elliptic regularity implies the existence of a  constant $c_1>0$ such that $$\sup_{m\in\N}|f_{\gamma_m}|_{C^{2,\alpha}}\leq c_1\sup_{m\in\N} (|f_{\gamma_m}|_{L^\infty}+|H_{g_0}(\Sigma_{g_t}(\gamma_m))|_{C^{0,\alpha}})=O(e^{-2/3t}).$$ 
			Thus $|A|^2_{g_0}(\Sigma_{g_t}(\gamma_m))=O(e^{-4/3t}+A_m^2)$, where $A_m:=||A_{g_0}(D(\gamma_m))||_{\infty}$. If $\lambda_j(g_t)$ and $\lambda_j(g_0)$, $j=1,2$ denote, respectively, the principal curvatures of $\Sigma_{g_t}(\gamma_m)$ with respect to $g_t$  and $g_0$, we have
			$$\lambda_j(g_t)=\lambda_j(g_0)+O(|h_t|_{C^1}),\quad j=1,2.$$ 
			Hence $|A|^2_{g_t}(\Sigma_{g_t}(\gamma_m))=O(e^{-4/3t}+A_m^2)$ and thus
			\begin{equation}\label{estimateA}
				2E_1=\frac{1}{4\pi(l_k-1)}\int_{{\bf \Sigma}_{g_t}(\gamma_m)}|A|_{g_t}^2dA_{g_t}=O(e^{-4/3t}+A_m^2).
			\end{equation}
			From Gauss equation, \eqref{gauss.phi}, and {the definition of $A_m$,} we obtain that
			$$|C(\phi_m)-1|= {O(A_m^2)}.$$ Moreover, from Proposition \ref{asymptotic.expn} we obtain 
			\begin{equation*}
				|Jac_{g_t}(F_{\gamma_m})|=1+O(|h_t|_{C^1})+O(|f_{\gamma_m}|_{C^1})=1+O(e^{-2/3t}).
			\end{equation*}
			Thus $|Jac_{g_t}(F_{\gamma_m}\circ\phi_m)|=|Jac_{g_t}(F_{\gamma_m})|C(\phi_m)=1+O(e^{-2/3t}+A_m^2)$ and this means  (see \eqref{jac.fcn}) that
			\begin{equation}\label{estimate.jac}
				\Lambda_{g_t}(\phi_m)=1+O(e^{-2/3t}+A_m^2).
			\end{equation}
			Recall the definitions in \eqref{omega.map} and  \eqref{theta.map}. {From Proposition \ref{asymptotic.expn}  and the fact that $|h_t|_{C^4}$ and  $\sup_{m\in\N}|f_{\gamma_m}|_{C^{2,\alpha}}$ have order $O(e^{-2/3t})$, we see that}{
				\begin{multline}\label{first.expn}
					[Ric(g_t)(\nu,\nu)-R(g_t)/3](\phi_m)\\
					=L(h_t)_{\phi_m(i)}(n(\gamma_m),n(\gamma_m))
					+O(e^{^{-4/3t}})\\
					=\theta(h_t)\circ\Omega(\phi_m)+O(e^{^{-4/3t}}).
			\end{multline}}
			Note that  $\theta(h_t)=O(|h_t|_{C^2})=O(e^{-2/3t})$ and so we obtain from \eqref{first.expn} and  \eqref{estimate.jac} that both
			$$ E_2=\delta_{\phi_m}([\theta(h_t)\circ\Omega+O(e^{^{-4/3t}})](\Lambda_{g_t}-1))=O(e^{^{-4/3t}}+A_m^2)$$
			and $E_3=O(e^{^{-4/3t}})$. As a result, 
			\begin{multline*}
				\delta_{\phi_m}\left(\left[Ric(g_t)(\nu,\nu)-{R(g_t)/3}+{|A|^2_{g_t}}/{2}\right]\Lambda_{g_t}\right)\\
				=\delta_{\phi_m}(\theta(h_t)\circ\Omega)+E_1+E_2+E_3=\Omega^*\delta_{\phi_m}(\theta(h_t))+O(e^{^{-4/3t}}+A_m^2)
			\end{multline*}
			and so we have from \eqref{fund.identity}
			$$\delta e^{-t}\leq  \Omega^*\delta_{\phi_m}(\theta(h_t))+O(e^{^{-4/3t}}+A_m^2).$$
			Making $m\to\infty$ the result follows because $A_m\to 0$.
		\end{proof}

		{Consider the operator on symmetric $2$-tensors 
			\begin{equation}\label{deturck}
				\mathcal A(h)=\Delta_{g_0} h-2(tr_{g_0}h )g_0+2h.
		\end{equation}}
		We know from Proposition \ref{ricci-stable} that $e^th_t$ converges in {$W^{k,2}$} to $\bar h$, where $\bar h$ satisfies $\mathcal A(\bar h)=-\bar h$. Thus  from Sobolev embedding theorem we have that $e^th_t$ converges to $\bar h$ in {$C^2$} and so $e^t\theta(h_t)\to \theta (\bar h)$. The previous lemma implies
		$$
		\delta\leq \bar\mu(\theta(\bar h)).
		$$
		We now show that $\bar \mu(\theta(\bar h))=0$, which gives us a contradiction and thus implies that for any sequence  $\Pi_m\in {S_{1/m,\bar\mu}}(M)$ we have
		$$\liminf_{m\to\infty}\frac{{\rm area}_g(\Pi_m)}{{\rm area}_{g_0}(\Pi_m)}\geq 1.$$

		\begin{prop}\label{ergodic}{We have $\bar \mu(\theta(\bar h))=0$}.
		\end{prop}
		\begin{proof}
			Because $\bar h$ is the lowest eigenfunction for $\mathcal A$, it follows from Koiso Bochner formula (see \cite{knopf}) that $\bar h$ is trace free and divergence free. Hence from the formula for the linearization of the  traceless Ricci tensor \cite[Theorem 1.174]{besse} we have
			$$L(\bar h)=-\frac{1}{2}\mathcal A(\bar h)=\frac{{\bar h}}{2}.$$
			Thus for all $(x,\{u_1,u_2,n\}) \in F(M)$ we have 
			\begin{equation}\label{theta.formula}
				\theta(\bar h)((x,\{u_1,u_2,n\}))=\frac{1}{2}\bar h_{x}(n,n).
			\end{equation}
			
			
			{We have $tr_{g_0}\bar h=0$  and  from the fact that $\bar\mu$ is the homogeneous Liouville measure on $F(M)$ we deduce
			$$\bar\mu(\theta(\bar h))=\frac{1}{\text{vol}(M)}\int_M \frac{1}{6}tr_{g_0}\bar h\,dV_{g_0}=0.$$ }
		\end{proof}

		\subsection{Proof of rigidity} Suppose that for some metric $g\in \mathcal V$ with $R(g)\geq -6$  and some  sequence  $\Pi_m\in {S_{1/m,\bar\mu}(M)}$ we have
		$$\liminf_{m\to\infty}\frac{{\rm area}_g(\Pi_m)}{4\pi(l_m-1)}= 1,$$
		where $l_m$ is the genus of an essential surface in $\Pi_m$. 
		Run normalized Ricci flow $(g_t)_{0\leq t\leq \bar t}$ starting at $g$ for a short time interval and set
		$$a(t):=\liminf_{m\to\infty}\frac{{\rm area}_{g_t}(\Pi_m)}{4\pi(l_m-1)}.$$
		From \eqref{loweRF} we see that  $a(0)=1$ implies that $a(t)\leq 1$ for all  $0\leq t\leq \bar t$ and thus $a(t)=1$ for all  $0\leq t\leq \bar t$.

		Suppose that $g$ is not Einstein. From the strong maximum principle applied to the evolution equation of $R(g_t)$ we obtain the existence of $\delta$ so that $R(g_t)\geq -6+2\delta$ for all $\bar t/2\leq t\leq\bar t$. Thus we see from \eqref{area.evolution} that  for all $\bar t/2\leq t\leq\bar t$ and all $m\in\N$ 
		$$
		\frac{d}{dt}\text{area}_{g_t}(\Pi_m)  \leq 4\pi(l_m-1)-(1+\delta)\text{area}_{g_t}(\Pi_m).
		$$
		ODE comparison gives us a contradiction because
		$$a(\bar t)\leq a(\bar t/2)e^{-(1+\delta)\bar t/2}+\frac{1-e^{-(1+\delta)\bar t/2}}{1+\delta}= e^{-(1+\delta)\bar t/2}+\frac{1-e^{-(1+\delta)\bar t/2}}{1+\delta}<1.$$
	\end{proof}
	
	\section{Proof of Theorem \ref{mse.thm}: Part II}
	We use Perelman's Ricci flow with surgery \cite{perelman} to remove the local condition on Theorem \ref{asympt.thm}.
	
	\begin{thm}\label{asympt.thm2} Assume $g$ is a metric on $M$ such that $R(g)\geq -6$. For any sequence  $\Pi_m\in {S_{1/m,\bar \mu}(M)}$ we have
		$$\liminf_{m\to\infty}\frac{{\rm area}_g(\Pi_m)}{4\pi(l_m-1)}\geq 1,$$
		where $l_m$ is the genus of an essential surface in $\Pi_m$.
		
		If equality holds then $g$ is isometric to $g_0$.
	\end{thm}
	\begin{proof}
		We will use the notation and results of \cite{perelman}. For a more detailed treatment the reader can consult \cite{bamler}.
		
		Note that Ricci flow and normalized Ricci flow \eqref{RF} differ only by scaling. With $\varepsilon_0$ small, Perelman finds in  \cite[Section 4]{perelman}  a sequence of manifolds $M_k$ with $M_0=M$, a discrete set of times $\{t_k\}_{k\in\N_0}$ with $t_0=0$,  and a sequence of smooth solutions to normalized Ricci flow $(\bar g_t)_{t_k\leq t< t_{k+1}}$ on $M_k$ with $\bar g_0=g$ so that the following holds: each $M_{k+1}$ is obtained from $M_k$ by surgery as described in \cite[Section 4.4]{perelman} but, because $M$ is irreducible and contains no embedded projective planes, the only possible surgeries remove $\varepsilon_0$-caps and glue in (perturbed) standard caps. Hence $M_k=M$ for all $k\in \N_0$ (after discarding  $3$-spheres) and, still following   \cite[Section 4]{perelman},  there is a sequence of compact sets $\Lambda_k\subset M$ such that $M\setminus \Lambda_k$ is diffeomorphic to a union of open balls, $\bar g_t$ converges smoothly to $\bar g_{t_{k+1}}$ on $\Lambda_k$ as $t\to t_{k+1}$, and for all $t<t_{k+1}$ close to $t_{k+1}$, the metric $\bar g_t$ is such that every $p$ in the boundary of $M\setminus \Lambda_k$ is  the center slice of some $\varepsilon_0$-neck. 
		
		
		In Section 7.1 of \cite{perelman} Perelman argues  that  $R(\bar g_t)\geq -6$ for all $t\geq 0$ (for the  normalized Ricci flow with surgery). In Section 7.3 and Section 7.4 of \cite{perelman}, Perelman shows that $(M,\bar g_t)$ admits, for all $t$ sufficiently large, a thick-thin decomposition, where the thick part meets the thin part along incompressible tori. Because $M$ is  closed and admits an hyperbolic metric, it has no incompressible tori and thus, for all $t$ sufficiently large, 
		$(M,\bar g_t)$ coincides with the thick part. Therefore, we obtain from Lemma 7.2 of \cite{perelman} that   $|Ric(\bar g_t)+2\bar g_t|_{C^0}$ is small  for  all $t$ sufficiently large. Thus we see from \cite{ye}  that for all $t$ sufficiently large $(\bar g_t)_{t\geq 0}$ is a smooth solution to normalized Ricci flow and converges to an hyperbolic metric  on $M$. In particular there are only finitely many surgery times $t_1<t_2<\ldots<t_K$. Using Mostow rigidity  we can  assume that, after applying a diffeomorphism to ${g_0}$,   $(\bar g_t)_{t\geq 0}$ converges smoothly to $g_0$ as $t\to\infty$. 
		
		Given a homotopy class $\Pi$ of essential surfaces, there is $\Sigma_t\in \Pi$, a minimal surface with respect to $\bar g_t$, such that ${\rm area}_{\bar g_t}(\Pi)={\rm area}_{\bar g_t}(\Sigma_t)$.
		
		\begin{lemm}We can find $\varepsilon_0$ small so that for every homotopy class $\Pi$ of essential surfaces  and  every  $t<t_{k+1}$ close to $t_{k+1}$, $\Sigma_t\subset \Lambda_k$.
		\end{lemm}
		\begin{proof}
			For all $t<t_{k+1}$ close to $t_{k+1}$, the metric $\bar g_t$ is such that every $p$ in the boundary of $M\setminus \Lambda_k$ is  the center slice of some $\varepsilon_0$-neck. This means that there exists a neighborhood $N\subset M$ of $p$ and a diffeomorphism (depending on $p$ and $t$)
			\[
			\Phi: S^2 \times [-{\varepsilon_0}^{-1}, {\varepsilon_0}^{-1}] \rightarrow N 
			\]
			such that $p\in \Phi(S^2 \times \{0\})$ and  for some $\lambda>0$ (depending on $p$ and $t$)
			\begin{equation} \label{closeness} 
				|\lambda^{-2}  \Phi^* {\bar g_{t}} - g_{S^2 \times [- {\varepsilon_0}^{-1}, {\varepsilon_0}^{-1}]}|_{C^{\lceil  1/{\varepsilon_0} \rceil}}< {\varepsilon_0}. 
			\end{equation} 
			Moreover, $\Phi(S^2 \times \{0\})$ is homotopic to a boundary component of $\Lambda_k$ because $M\setminus\Lambda_k$ consists of $\varepsilon_0$-caps.
			
			Suppose for contradiction that $\Sigma_{t}$ passes through $p$.  We first give the argument in the case that $\Sigma_t$ is embedded.  The region $M\setminus \Lambda_k$ is diffeomorphic to a disjoint union of balls and so there is a ball $B_k$ containing $N$ so that $\partial B_k=\Phi(S^2 \times \{- 1/{\varepsilon_0}\})$ and $\Phi(S^2 \times [-\epsilon_0^{-1}, \epsilon_0^{-1}] \subset B_k$. It cannot be the case that $\Sigma_{t}$ is contained in $B_k$ and thus $\Sigma_{t}$ must intersect $\partial B_k$.  We can perturb $B_k$ slightly so that $\partial B_k$ intersects $\Sigma_t$ transversely  in a union of circles.  Let $D_p$ be the connected component of  $\Sigma_t\cap B_k$ that contains $p$.  Then $D_p$ is homeomorphic to an $m$-holed sphere for some $m$, all of the connected components of whose boundary  $\gamma$ correspond to null-homotopic embedded loops in $\Sigma_t$.   
			
			On the one hand, $\gamma$ can be filled in by a region $D_{\gamma}$  in $\partial B_k$ with area at most roughly $4\pi\lambda^2$. On the other hand, $D_p$ must intersect every cross-section $\Phi(S^2 \times \{y \})$ for $- 1/{\varepsilon_0}<y<0$.  By the monotonicity formula, there is a universal constant $c$ such that for $y \in (-1/{\varepsilon_0}+1/2, -1/2)$, 	 
			\[
			\text{area}_{\bar g_t}(D_p \cap \Phi( S^2 \times (y-1/2,y+1/2)) > c    \lambda^{2}.    
			\]
			It follows by choosing disjoint unit intervals in $(-{1}/{\varepsilon_0}+1/2, -1/2)$ that if we chose ${\varepsilon_0}$ such that $1/{\varepsilon_0}$ is greater than $5\pi/c$, then the area of $D_p$ will be larger than $(4\pi+1)\lambda^2$.  By cutting out $D_p$ and gluing in $D_{\gamma} \subset \partial B_k$ we could then produce a surface homotopic to $\Sigma_{t}$ but with smaller area, which is a contradiction.  
			
			In the case that $\Sigma_t$ is only immersed, we lift $\Sigma_t$ to a surface $\overline{\Sigma}_t$ in the cover $M_{\Sigma_t}$ of $M$ homeomorphic to $\Sigma_t \times \mathbb{R}$ corresponding to the inclusion of $\pi_1(\Sigma_t) \rightarrow \pi_1(M)$ for some choice of basepoint on $\Sigma_t$.  Since $\Sigma_t$ minimizes area in its homotopy class in $M$,  $\overline{\Sigma}_t$ minimizes area in  its homotopy class in $M_{\Sigma_t}$.  It then follows from \cite{fhs}[Theorem 2.1] that $\overline{\Sigma}_t$ is embedded in $M_{\Sigma_t}$.  
			
			Assume now that $\Sigma_t$ passes through a point $p$ contained in some ball $B_k$ as above.  Since  $B_k$ is contractible, it lifts to  $M_{\Sigma_t}$, and we can find a lift $\overline{B}_k$ of $B_k$ containing a lift of $p$ lying on $\overline{\Sigma}_t$.  Since $\overline{\Sigma}_t$ is embedded, we can apply the above arguments to get a contradiction as before.

		\end{proof}
		This lemma implies that we can use \cite[Lemma 9]{bben} and conclude that $\bar A(t):={\rm area}_{\bar g_t}(\Pi)$ is a  Lipschitz  function and thus differentiable almost everywhere. Let  $\bar t$ be a  point of differentiability of $\bar A$ and consider the function $a(t):= {\rm area}_{\bar g_t}(\Sigma_{\bar t})$. We have $\bar A(t)\leq a(t)$ for all $t$ and $\bar A(\bar t)=a(\bar t)$. Hence, arguing like in \eqref{area.evolution} we deduce
		$$\bar A'(\bar t)=a'(\bar t)\leq -2\pi\chi(\Sigma_{\bar t})-\bar A(\bar t).$$
		From this ODE we can argue like in the proof of Theorem \ref{asympt.thm} and conclude that if  for some sequence  $\Pi_m\in {S_{1/m,\bar\mu}(M)}$ we have $$\frac{{\rm area}_g(\Pi_m)}{4\pi(l_m-1)}\leq 1-\delta,$$
		where $l_m$ is the genus of an essential surface in $\Pi_m$, then  for all $t\geq 0$
		$$\frac{{\rm area}_{\bar g_t}(\Pi_m)}{4\pi(l_m-1)}\leq 1-\delta e^{-t}.$$
		We have that $\bar g_t$ converges smoothly to $g_0$ and so this contradicts Theorem \ref{asympt.thm}. Thus the inequality in Theorem \ref{asympt.thm2} must hold.
		
		If equality holds in Theorem \ref{asympt.thm2}, the very same argument used as in Theorem \ref{asympt.thm} shows that $g_0$ must be Einstein.
	\end{proof}
	
	We can now derive Theorem \ref{mse.thm} from Theorem \ref{asympt.thm2}.
	\begin{thm}
		Suppose $(M,g)$ has $R(g)\geq -6$. Then ${E_{\bar\mu}(g)}\leq 2$ and equality occurs if and only if $g$ is hyperbolic.
	\end{thm}
	\begin{proof} In what follows we will use the fact (from \cite{seppi}) that for any $\Pi_m\in S_{1/m}(M)$, if $l_m$ is the genus of an essential surface in $\Pi_m$, then we have ${\rm area}_{g_0}(\Pi_m)/{4\pi(l_m-1)}\to 1$ as $m\to\infty$.
		
		Choose $\delta>0$. Theorem \ref{asympt.thm2}   implies the the existence of $\varepsilon_0$ so that for all $\Pi\in {S_{\varepsilon_0,\bar\mu}(M)}$  we  have
		$${\rm area}_{g_0}(\Pi)\leq(1+\delta){\rm area}_g(\Pi).$$
		Hence for all $\varepsilon\leq \varepsilon_0$ we have
		\begin{multline*}
			\#\{\text{area}_g(\Pi)\leq 4\pi(L-1):\Pi\in {S_{\varepsilon,\bar\mu}(M)}\}\\
			\leq \#\{\text{area}_{g_0}(\Pi)\leq (1+\delta)4\pi(L-1):\Pi\in {S_{\varepsilon,\bar\mu}(M)}\}.
		\end{multline*}
		This inequality and the expression \eqref{asymp.area} for the minimal surface entropy imply
		$${E_{\bar\mu}(g)\leq E_{\bar\mu}(g_0)(1+\delta)}=2(1+\delta).$$
		The inequality follows from making $\delta\to 0$. 
		
		Suppose now that ${E_{\bar\mu}(g)}=2$ for some metric with $R(g)\geq -6$. Reasoning as above we see that if we could find $\delta>0$ and $\varepsilon_0$ so that 
		for all $\Pi\in {S_{\varepsilon_0,\bar\mu}(M)}$
		$${\rm area}_{g_0}(\Pi)\leq (1-\delta){\rm area}_g(\Pi),$$
		then ${E_{\bar\mu}(g)}\leq 2(1-\delta)$.  Hence there is $\Pi_m\in {S_{1/m,\bar\mu}(M)}$ so that 
		$$\liminf_{m\to\infty}\frac{{\rm area}_g(\Pi_m)}{{\rm area}_{g_0}(\Pi_m)}\leq 1$$
		and Theorem \ref{asympt.thm2} implies that $g$ is hyperbolic.
	\end{proof}
	
	\section{Proof of Theorem \ref{rel.thm}}

	For each $m\in\N$, Labourie \cite[Theorem 5.7]{Labourie}  found $N_m\in\N$, $\{\phi^i_m\}_{i=1}^{N_m}$ in  $\mathcal F(M,1/m)$, and $ 0<a_m^1,\ldots,a_m^{N_m}<1$ with $a_m^1+\ldots+a_m^{N_m} =1$, so that:  
	$\phi^i_m$ is equivariant  with respect to a representation of a Fuchsian subgroup  of ${\rm PSL}(2,\R)$ in $\pi_1(M)<\text{PSL}(2,\C)$ and 
	the  laminar measure 
	\begin{equation}\label{measure.laminar}
		\delta_m:=\sum_{i=1}^{N_m}a^i_m\delta_{\phi^i_m}
	\end{equation} is such that
	$\Omega_*\delta_{m}$ converges to  $\bar\mu$, the unit {Liouville} homogeneous measure on $F(M)$.
	
	\begin{prop}\label{connected}
		There is $\phi_m$ in  $\mathcal F(M,1/m)$ equivariant  with respect to a representation of a Fuchsian group $\Gamma_m$  of ${\rm PSL}(2,\R)$ in $\pi_1(M)<\text{PSL}(2,\C)$ such that  $\Omega_*\delta_{\phi_m}$ converges to  ${\bar\mu}$ as $m\to\infty$.
	\end{prop}
	\begin{proof}
		The space of all closed totally geodesic immersions in $(M,g_0)$ is countable and so 
		$$\mathcal T:=\{F(S)\subset F(M):\mbox{S is a closed totally geodesic immersion in }(M,g_0)\} $$
		is also countable, where $F(S)$ denotes the frame bundle of $S$ which injects naturally in $F(M)$.
		
		Considering tubular neighborhoods, we can find a decreasing sequence of open sets $\{B_k\}_{k\in\N}\subset F(M)$ so that for all $k\in\N$
		$$\cup_{T\in\mathcal T}T\subset B_k,\quad {\bar\mu}(\partial B_k)=0,\quad\mbox{and}\quad{\bar\mu}(B_k)\leq 2^{-2k-3}.$$ 
		\begin{lemm}\label{b.is.sero} For each $j\in\N$, there is $j\leq m_j\in\N$  and $\phi_j\in \{\phi^1_{m_j},\ldots,\phi^{N_{m_j}}_{m_j}\}$ so that 
			$$\Omega_*\delta_{\phi_j}(B_k)\leq 2^{-(k+1)}\quad\mbox{for all }k\leq j.$$
		\end{lemm} 
		\begin{proof}
			For all $k\in\N$, $\Omega_*\delta_m(B_k)\to{\bar\mu}(B_k)$ as $m\to\infty$. Thus we can find a strictly increasing sequence of integers $\{m_j\}_{j\in\N}$ so that
			\begin{equation}\label{connected1}
				\Omega_*\delta_{m_j}(B_k)\leq 2{\bar\mu}(B_k)\leq 2^{-2(k+1)}\quad\mbox{for all }k\leq j.
			\end{equation}
			Relabel $m_j$ to be $j$ so that $\delta_{m_j}, \phi^i_{m_j}, N_{m_j},$ and $a^i_{m_j}$  become $\delta_j, \phi^i_j$, $N_{j}$, and $a^i_{j}$, respectively.
			
			Consider $\mu_j$ a unit measure on $\{1,\ldots,N_j\}$ so that $\mu_j(i)=a^i_j$ and set
			$$J_{j,k}:=\{i\in \{1,\ldots,N_j\}: \Omega_*\delta_{\phi^i_j}(B_k)\geq 2^{-(k+1)}\}.$$
			From \eqref{connected1} and the definition of $J_{j,k}$ we have that for all $j\geq k$
			$$
			\mu_j(J_{j,k})=\sum_{i\in J_{j,k}}a^i_j\leq 2^{k+1}\sum_{i\in J_{j,k}}a^i_j \Omega_*\delta_{\phi^i_j}(B_k)
			\leq  2^{k+1}\Omega_*\delta_j(B_k)\leq 2^{-(k+1)}.
			$$
			Thus, for all $j\in\N$ and $k\leq j$,
			$\mu_j(\cup_{k=1}^jJ_{j,k})\leq \sum_{k=1}^j2^{-(k+1)}\leq 1/2.$
			Hence $$A_j:=\{1,\ldots,N_j\}\setminus\cup_{k=1}^jJ_{j,k}\neq \emptyset $$
			and we pick $l_j\in A_j$. The maps $\phi_j:=\phi^{l_j}_{m_j}$ satisfy the desired conditions.
		\end{proof}

		After passing to a subsequence, $\Omega_*\delta_{\phi_j}$ converges weakly to a unit measure $\nu$ on $F(M)$. From Lemma \ref{compactness.lemm} we have that $\nu$ is  invariant under  the homogeneous action of $\text{PSL}(2,\R)$.  The proof will be completed if we show that $\nu={\bar\mu}$.
		
		Ratner's classification theorem \cite{ratner} implies that every ergodic probability measures on $F(M)$ invariant under  the homogeneous action of $\text{PSL}(2,\R)$ is either ${\bar\mu}$ or is supported in some $T\in \mathcal T$. Thus, from the ergodic decomposition theorem for $\text{PSL}(2,\R)$-actions  (Theorem 4.2.6 in \cite{katok}), there is $0\leq \theta\leq 1$ so that the measure $\nu$ decomposes as  $\nu=\theta{\bar\mu}+(1-\theta)\mu_{\mathcal T}$, where $\mu_{\mathcal T}$ is some  probability measure on $F(M)$ with support in $\cup_{T\in\mathcal T}T$. Moreover we have from 
		Lemma \ref{b.is.sero} that $\nu(B_k)\leq 2^{-(k+1)}$ for all $k\in\N$. 
		
		Recall that $\cup_{T\in\mathcal T}T\subset B_{k}$ for all $k \in\N$ and so
		$$1-\theta=(1-\theta)\mu_{\mathcal T}(B_k)\leq  \nu(B_k)\leq 2^{-(k+1)}.$$
		Hence $\theta=1$, which means that $\nu={\bar\mu}$.
	\end{proof}

	\begin{prop}\label{compute}   
		$ {\rm Area}_{\rm Id}(g/g_0)=\lim_{m\to\infty}\frac{{\rm area}_g({\bf D}(\partial\phi_m))}{2\pi|\chi({\bf D}(\partial\phi_m))|}.$
	\end{prop}
	\begin{proof}
		We have for all $(x,P)\in Gr_2(M)$
		$$|\wedge^2{\rm Id}^{-1}|_g(x,P)=|\wedge^2{\rm Id}|^{-1}_g(x,P).$$
		We abuse notation and also denote by $|\wedge^2{\rm Id}|_g$ the following smooth positive function on the frame bundle $F(M)$ 
		$$|\wedge^2{\rm Id}|_g: F(M)\rightarrow \R, \quad (y,\{e_1,e_2,n\})\mapsto|\wedge^2{\rm Id}|_g(y,{\rm span}\{e_1,e_2\}).$$
		We have
		$${\rm Area}_{\rm Id}(g/g_0) ={\bar\mu}(|\wedge^2{\rm Id}|^{-1}_g)$$
		and thus the fact that $\Omega_*\delta_{\phi_m}\to {\bar\mu}$ implies
		\begin{equation}\label{Area.final}
			{\rm Area}_{\rm Id}(g/g_0) ={\bar\mu}(|\wedge^2{\rm Id}|^{-1}_g)=\lim_{m\to\infty}\Omega_*\delta_{\phi_m}(|\wedge^2{\rm Id}|^{-1}_g).
		\end{equation}Consider $U_m\subset \text{PSL}(2,\R)$ a fundamental domain of $\text{PSL}(2,\R)/ \Gamma_m$. Then
		$$\Omega_*\delta_{\phi_m}(|\wedge^2{\rm Id}|^{-1}_g)=\frac{1}{vol(U_m)}\int_{U_m} |\wedge^2{\rm Id}|^{-1}_g\circ\Omega(\phi_m\circ \tau)d\bar\nu(\tau).$$
		Note that for all  $\tau\in  \text{PSL}(2,\R)$  we have, with $z=\tau(i)$,
		$$|\wedge^2{\rm Id}|_g\circ\Omega(\phi_m\circ \tau)=|\wedge^2{\rm Id}|_g(\phi_m(z),(d\phi_m)_{z}(T_z\H^2)).$$
		Thus, denoting the hyperbolic volume form on $\H^2$ by $dA_{hyp}$,
		\begin{multline}\label{lemm.formula.area}
			\Omega_*\delta_{\phi_m}(|\wedge^2{\rm Id}|^{-1}_g)\\
			=\frac{1}{2\pi|\chi({\bf D}(\partial\phi_m))|}\int_{\H^2/ \Gamma_m} |\wedge^2{\rm Id}|^{-1}_g(\phi_m(z),(d\phi_m)_{z}(T_z\H^2))dA_{hyp}(z)\\
			=\frac{1}{2\pi|\chi({\bf D}(\partial\phi_m))|}\int_{{\bf D}(\partial\phi_m)}\frac{|\wedge^2{\rm Id}|^{-1}_g(y,T_y{\bf D}(\partial\phi_m))}{C(\phi_m)\circ \phi_m^{-1}(y)}dA_{g_0}(y).
		\end{multline}
		Set $A_m:=||A||_{L^\infty({\bf D}(\partial\phi_m))}$.
		From Gauss equation and \eqref{gauss.phi} we have
		\begin{equation}\label{inequality}
			1\leq \frac{1}{C(\phi_m)}\leq 1+A_m^2.
		\end{equation}
		The co-area formula says that
		$$\int_{{\bf D}(\partial\phi_m)}{|\wedge^2{\rm Id}|^{-1}_g(y,T_y{\bf D}(\partial\psi_m))}{dA_{g_0}}={\rm area}_g({\bf D}(\partial\phi_m))$$
		and thus, combining  with \eqref{lemm.formula.area} and \eqref{inequality}, we obtain
		$$\frac{{\rm area}_g({\bf D}(\partial\phi_m))}{2\pi|\chi({\bf D}(\partial\phi_m))|}\leq \Omega_*\delta_{\phi_m}(|\wedge^2{\rm Id}|^{-1}_g)\leq (1+A_m^2)\frac{{\rm area}_g({\bf D}(\partial\phi_m))}{2\pi|\chi({\bf D}(\partial\phi_m))|}.$$
		We have that $A_m\to 0$ as $m\to\infty$ and hence we deduce from this inequality  and \eqref{Area.final} the desired result.
	\end{proof}

	\subsection{Proof of inequality} The closed surfaces  ${\bf D}(\partial\phi_m)$ define a homotopy class $\Pi_m\in S_{1/m}(M)$ for all $m\in\N.$  We denote the genus of ${\bf D}(\partial\phi_m)$ by $l_m$. {From Proposition \ref{connected} we have that $\mu_{\Pi_m}\to\bar\mu$ and so after relabeling $\{\Pi_m
	\}_{m\in\N}$ we have that $\Pi_m\in S_{1/m,\bar\mu}(M)$ for all $m\in\N$}.
	
	\begin{prop}${\rm Area}_{{\rm Id}}(g/g_0) {E_{\bar\mu}(g)}\geq 2$.
	\end{prop}
	\begin{proof}
		After passing to a subsequence, set
		$$\alpha:=\lim_{m\to\infty}\frac{{\rm area}_g(\Pi_m)}{4\pi(l_m-1)}.$$
		Given $\delta$ we have that for all $m$ sufficiently large
		$${\rm area}_g(\Pi_m)\leq (\alpha+\delta)4\pi(l_m-1).$$
		Let ${\bf D}_m^k$, $l_m^k$, and $\Pi_m^k\in {S_{1/m,\bar\mu}(M)}$ denote, respectively, a $k$-cover of  ${\bf D}(\partial\phi_m)$,  its genus, and  its homotopy class.
		From the inequality above we have for all $k\in\N$ and all $m\in\N$ sufficiently large
		\begin{equation}\label{bound.below.area}
			\frac{{\rm area}_g(\Pi_m^k)}{4\pi(l_m^k-1)}\leq \frac{k\,{\rm area}_g(\Pi_m)}{4\pi(l_m^k-1)}=\frac{{\rm area}_g(\Pi_m)}{4\pi(l_m-1)}\leq \alpha+\delta.
		\end{equation}
		From the   M\"uller-Puchta's formula (see \cite[Section 4]{cal-marques-neves}) there is $c(m)>0$ so that  ${\bf D}(\partial\phi_m)$ has at least $(c(m)l_m^k)^{2l_m^k}$ distinct covers of degree less than or equal to $k$.  Thus, if we choose $L_m^k$ so that $4\pi(L_m^k-1)= (\alpha+\delta)4\pi(l^k_m-1)$, we have from \eqref{bound.below.area} that for all $k\in \N$ and all $m\in\N$ sufficiently large
		$$\#\{\text{area}_g(\Pi)\leq 4\pi(L_m^k-1):\Pi\in {S_{1/m,\bar\mu}(M)}\}\geq (c(m)l_m^k)^{2l_m^k}.$$
		Hence for all $m$ sufficiently large
		\begin{multline*}
			\lim_{k\to\infty}\frac{\ln \#\{\text{area}_g(\Pi)\leq 4\pi(L_m^k-1):\Pi\in{S_{1/m,\bar\mu}(M)}\}}{L_m^k\ln L_m^k}\\
			\geq \lim_{k\to\infty}\frac{2l_m^k\ln(c(m)l_m^k)}{L_m^k\ln L_m^k}=\frac{2}{\alpha+\delta}.
		\end{multline*}
		From  the expression for ${E_{\bar\mu}(g)}$ in \eqref{asymp.area} we obtain that
		$(\alpha+\delta)E_{\bar\mu}(g)\geq 2$. Making $\delta\to 0$ we deduce $\alpha {E_{\bar\mu}(g)}\geq 2$. From Proposition \ref{compute} we have
		$$ {\rm Area}_{{\rm Id}}(g/g_0)=\lim_{m\to\infty}\frac{{\rm area}_g({\bf D}(\partial\phi_m))}{2\pi|\chi({\bf D}(\partial\phi_m))|}\geq \lim_{m\to\infty}\frac{{\rm area}_g(\Pi_m)}{4\pi(l_m-1)}=\alpha$$
		and thus ${\rm Area}_{{\rm Id}}(g/g_0) {E_{\bar\mu}(g)}\geq 2$.
	\end{proof}
	
	\subsection{Proof of rigidity} Suppose that ${\rm Area}_{{\rm Id}}(g/g_0) {E_{\bar\mu}(g)}=2$. In this case, the proof of the preceding  proposition implies that
	
	\begin{equation}\label{rigidity.area}
		\lim_{m\to\infty}\frac{{\rm area}_g({\bf D}(\partial\phi_m))}{{\rm area}_g(\Pi_m)}=1.
	\end{equation}
	The mean curvature of a surface in $(M,g)$  is denoted by $H_g$.
	\begin{lemm}
		$\lim_{m\to\infty}\frac{1}{{\rm area}_g({\bf D}(\partial\phi_m))}\int_{{\bf D}(\partial\phi_m)}|H_g|^2dA_g=0.$
	\end{lemm}
	\begin{proof}
		Suppose that, after passing to a subsequence, there is $\delta>0$ so that
		\begin{equation}\label{contradiction}
			\int_{{\bf D}(\partial\phi_m)}|H_g|^2dA_g\geq 2\delta\, {\rm area}_g({\bf D}(\partial\phi_m))
		\end{equation}
		for all $m\in\N$.
		
		The second fundamental form of ${\bf D}(\partial\phi_m)$ (with respect to $g$) and any of its derivatives are uniformly bounded independently of $m$. Arguing like in \cite[Corollary 4.4]{ecker-huisken} we find  $t_0>0$, $C_0>0$, and $\{{\bf D}_m(t)\}_{0\leq t\leq t_0}$ a solution to mean curvature flow with initial condition ${\bf D}_m(0):={\bf D}(\partial\phi_m)$ so that both the mean curvature of ${\bf D}_m(t)$ and its derivative  is bounded uniformly by $C_0$. 
		
		With $H_g(t)$, $A_g(t)$, and $\nu_t$ denoting, respectively, the mean curvature, second fundamental form, and normal vector of ${\bf D}_m(t)$, we have  the evolution equation
		$$\frac{d}{dt}H_g(t)= \Delta H_g(t)+(|A_g(t)|^2+Rc(g)(\nu_t,\nu_t))H_g(t).$$
		Thus there is a constant $C_1>0$ (depending only on $C_0$ and $g$) so that for all $m\in \N$ and $0\leq t\leq t_0$
		$$\frac{d}{dt}\int_{{\bf D}_m(t)}|H_g(t)|^2dA_g \geq -C_1{\rm area}_g({\bf D}_m(t))\geq -C_1{\rm area}_g({\bf D}(\partial\phi_m)).$$
		Choose $t_1<\min\{{\delta/C_1}, t_0\}$. From \eqref{contradiction} and the inequality  above we have that for all  $m\in \N$ and $0\leq t\leq t_1$
		$$\int_{{\bf D}_m(t)}|H_g|^2dA_g\geq \delta\,{\rm area}_g({\bf D}(\partial\phi_m))\geq \delta\, {\rm area}_g({\bf D}_m(t)).$$
		Therefore
		$$ \frac{d}{dt}{\rm area}_g({\bf D}_m(t))=-\int_{{\bf D}_m(t)}|H_g|^2dA_g\leq -\delta\, {\rm area}_g({\bf D}_m(t))$$
		and thus for all $m\in \N$ we have
		$${\rm area}_g(\Pi_m)\leq  {\rm area}_g({\bf D}_m(t_1))\leq e^{-\delta t_1}{\rm area}_g({\bf D}(\partial\phi_m)).$$
		This contradicts \eqref{rigidity.area}. 
	\end{proof}
	
	Recall that $\Omega^*\delta_{\phi_m}$ converges to ${\bar \mu}$ as $m\to\infty$.  The previous lemma allows us to apply Theorem \ref{CMN.thm} to  conclude  the following:
	
	For every $\gamma\in\mathcal C_0$, there is a lift $D_m\subset \H^3$  of ${\bf D}(\partial\phi_m)\subset M$ so that, after passing to a subsequence, 
	$D_m$  converges on compact sets to the totally geodesic disc $D(\gamma)\subset \H^3$ and
	$$
	\lim_{m\to\infty}\int_{D_m\cap B_R(0)}|H_g|^2dA_{g_0}=0\quad\mbox{for all }R>0.
	$$
	Thus we obtain that $D(\gamma)$ is a minimal disc for the metric $g$. The arbitrariness of $\gamma$ implies that every totally geodesic disc of $\H^3$ is minimal with respect to $g$. 
	
	The next theorem follows from adapting some of the arguments  in \cite{amb-marques-neves}. 
	\begin{thm} Let $g$ be a metric on $M$ with property that every totally geodesic disc with respect to $g_0$ is minimal with respect to $g$. There is a constant   $c>0$ so that $g=cg_0$.
	\end{thm}
	\begin{proof}

		We denote the space of Killing symmetric $2$-tensors on a Riemannian manifold $(X,h)$ by $ K_2(X,h)$. They are characterized by the property (see \cite[Section 1]{Tak}) that a symmetric $2$-tensor $k$ lies in $K_2(X,h)$ if and only if for every geodesic $\gamma \subset X$ the  function below is constant
		$$t\mapsto k(\gamma'(t),\gamma'(t)).$$
		Given a symmetric $2$-tensor $k$ on $X$, we denote the absolute value of its determinant with respect to $h$ by $|k|_h$.
		
		Let $\delta$ denote the Euclidean metric in $\R^3$ and $B\subset \R^3$ denote the open unit ball. Consider the Beltrami-Klein model $(B,g_{\rm hyp})$ for $(\H^3,g_0)$. The important property of this model is that the image of geodesics and totally geodesic discs in $(B,g_{\rm hyp})$  is the same as affine lines and affine planes in $(B,\delta).$  
		
		In  \cite[Theorem 9.8]{amb-marques-neves} the authors classified all metric $h$ on $S^n$  for which every totally geodesic hypersphere is minimal with respect to $h$. They are those for which
		$$|h|_{g_{\rm round}}^{-2/(n+1)}h\in K_2(S^n,g_{\rm round}).$$
		A similar reasoning, which we repeat for the sake of completeness, gives
		\begin{prop}
			A positive definite symmetric $2$-tensor $h$ on $B$ has the property that every affine plane   is minimal with respect to $h$ if and only if  $$|h|_{g_{\rm hyp}}^{-1/2}h\in K_2(B,g_{\rm hyp}).$$
		\end{prop}
		\begin{proof}
			We will use the following theorem of Hangan \cite{Han1} :
			A metric $h$ on $B$ is such that all affine planes have zero mean curvature with respect to $h$ if and only if
			$$|h|_{\delta}^{-1/2}h\in K_2(B,\delta).$$
			\begin{lemm} With $k$ a $2$-tensor on $B$ then
				$$k\in K_2(B,g_{\rm hyp})\iff  |g_{\rm hyp}|_{\delta}^{-1/2}k\in K_2(B,\delta)$$
			\end{lemm}
			\begin{proof}
				From Hangan's theorem we have that $|g_{\rm hyp}|_{\delta}^{-1/2}g_{\rm hyp}\in K_2(B,\delta)$ and hence,  for every Euclidean geodesic $\gamma\subset B$, we  have some constant $c_0=c_0(\gamma)$ such that for all $t$
				$$g_{\rm hyp}(\gamma'(t),\gamma'(t))=|g_{\rm hyp}|_{\delta}^{1/2}(\gamma(t))c_0.$$
				Note that geodesics with respect to $g_{\rm hyp}$ are also straight lines. Hence  if  we perform a change of variable according to
				$$\frac{dt}{ds}(s)=|g_{\rm hyp}|_{\delta}^{-1/4}(\gamma(t(s)))$$ we have that $s\mapsto \sigma(s):=\gamma(t(s))$ is a geodesic for $g_{\rm hyp}$.  Therefore
				$$
				\frac{d}{ds}k(\sigma'(s),\sigma'(s))=\frac{dt}{ds}(s)\frac{d}{dt}_{t=t(s)}\left(\frac{k(\gamma'(t),\gamma'(t))}{|g_{\rm hyp}|_{\delta}^{1/2}(\gamma(t))}\right)
				$$
				and so
				$$
				\frac{d}{ds}k(\sigma'(s),\sigma'(s))=0 \iff \frac{d}{dt}_{t=t(s)}\left(\frac{k(\gamma'(t),\gamma'(t))}{|g_{\rm hyp}|_{\delta}^{1/2}(\gamma(t))}\right)=0.
				$$
				Thus $k$ is constant along hyperbolic geodesics if and only if  $|g_{\rm hyp}|_{\delta}^{-1/2}k$ is constant along Euclidean geodesics, which implies the result.
			\end{proof}
			Using this lemma and the identity $|h|_{\delta}=|h|_{g_{\rm hyp}}|g_{\rm hyp}|_{\delta}$ we have
			$$ |h|_{\delta}^{-1/2}h\in K_2(B,\delta) \iff |h|_{g_{\rm hyp}}^{-1/2}h\in K_2(B,g_{\rm hyp}).$$
			Hangan's theorem \cite{Han1} implies the desired result.
		\end{proof}

		Denote the lift of the metric $g$ to $\H^3$ by $g$ as well. We are assuming that every totally geodesic disc in $\H^3$ is minimal with respect to $g$ and  so the previous proposition implies that  $G:=|g|_{g_0}^{-1/2}g\in  K_2(M,g_0)$.  
		
		The geodesic flow in $(M,g_0)$ is ergodic and so we can choose a geodesic $\sigma\subset M$ which is dense in the unit tangent bundle. From the fact that $t\mapsto G(\sigma'(t),\sigma'(t))$ is constant we deduce the existence of a constant $\alpha$ so that  $G(Y,Y)=\alpha g_0(Y,Y)$ for every vector  field $Y$. This implies that $G=\alpha g_0$. Using the fact that $G=|g|_{g_0}^{-1/2}g$ we deduce that $g=\alpha^{-2}g_0$.

	\end{proof}

	\appendix
	
	\section{Auxiliary Results}
	
	\subsection{Asymptotic Expansion}
	
	With $g\in \mathcal U$, set $h:=g-g_0$. Without loss of generality we assume that $|h|_{C^4}\leq 1$.  
	
	Given $\gamma\in \mathcal C_{\bar \varepsilon}$, Recall that $\nu_g(\gamma)$ denotes a  continuous unit normal vector field along $\Sigma_g(\gamma)$ with respect to $g$ and $n(\gamma)$ denotes a  continuous unit normal vector field along $D(\gamma)$ with respect to $g_0$.
	
	Recall the diffeomorphism (using the hyperboloid model)
	\begin{equation*}
		F_{\gamma}: D(\gamma)\rightarrow \Sigma_{g}(\gamma),\quad x\mapsto \cosh(f_{\gamma}(x))x+\sinh(f_{\gamma}(x))n(\gamma)(x).
	\end{equation*}

	Given $p\in M$ and $k\in \N_0$, $l\in \N$, we denote by $O_p^k(|h|^l)$ any quantity for which there is a constant $\alpha_{k,l}$ (independent of $p$ and $g\in \mathcal U$) so that its absolute value at $p$ is bounded by $\alpha_{k,l}\sum_{j=0}^k|\nabla^j h|^l(p)$.  Likewise, given $\gamma\in\mathcal C_{\bar \varepsilon}$ and $x\in D(\gamma)$, we denote  by $O_x^k(|f_{\gamma}|^l)$ any quantity for which there is a constant $\beta_{k,l}$ (independent of $x$, $\gamma$, and $g\in \mathcal U$) so that its absolute value at $x$ is bounded by $\alpha_{k,l}\sum_{j=0}^k|\nabla^j f_{\gamma}|^l(x)$. 
	
	Let $L$ denote the linearization at $g_0$ of 
	$$h\mapsto \mathring{Ric}(g_0+h):=Ric(g_0+h)-\frac{R(g_0+h)}{3}(g_0+h).$$
	
	\begin{prop}\label{asymptotic.expn}
		With $\gamma\in \mathcal C_{\varepsilon}$, $x\in D(\gamma)$ and $y=F_{\gamma}(x)$ we have
		\begin{multline*}Ric(g)_y(\nu_g(\gamma),\nu_g(\gamma))-\frac{1}{3}R(g)(y)=L(h)_x(n(\gamma),n(\gamma))\\ +O^3_x(|h|^2)+O^1_x(|f_{\gamma}|^2).
		\end{multline*}
		With $|Jac_g F_{\gamma}|$ the Jacobian of $F_{\gamma}:(D(\gamma),g_0)\rightarrow (\Sigma_{g}(\gamma),g)$ we have
		$$|Jac_g F_{\gamma}|(x)=1+O_x^1(|h|)+O_x^1(|f_{\gamma}|).$$
	\end{prop}
	
	\begin{proof}
		Set $$X:=\{(\gamma,f):\gamma\in\mathcal C_{\bar \varepsilon}, f\in C^2(D(\gamma)), |f|_{C^2}\leq 1\}.$$
		Given $(\gamma,f)\in X$ set
		\begin{equation*}
			F(\gamma,f): D(\gamma)\rightarrow \Sigma_{g}(\gamma),\quad x\mapsto \cosh(f(x))x+\sinh(f(x))n(\gamma)(x)
		\end{equation*}
		and $\Sigma(\gamma,f):=F(\gamma,f)(D(\gamma))$. Let $n(\gamma,f)$ be the unit normal vector field with respect to $g_0$ along $\Sigma(\gamma,f)$, defined so that $n(\gamma,f)$ depends smoothly on $(\gamma,f)$ and $n(\gamma,0)=n(\gamma)$.
		We have for all $0\leq t\leq 1$
		$$\left|\frac{d^k}{(dt)^k}(n(\gamma,tf)(F(\gamma,tf)(x))\right|=O_x^1(|f|^k), \quad k=1,2.$$
		Hence, if $G$ is a $2$-tensor on $M$ with $|G|_{C^2}\leq 1$, then
		$$\alpha(t):=G(n(\gamma,tf)(F(\gamma,tf)(x)),n(\gamma,tf)(F(\gamma,tf)(x)))$$
		satisfies 
		$$|\alpha'(0)|\leq (|\nabla G|(x)+|G|(x))O_x^1(|f|)\quad\mbox{and}\quad \sup_{0\leq t\leq 1}|\alpha''(t)|=O_x^1(|f|^2).$$
		Therefore, with $y=F(\gamma,f)(x)$, we obtain from Taylor's expansion
		\begin{multline}\label{taylor.epn}
			|G(n(\gamma,f)(y),n(\gamma,f)(y))-G(n(\gamma)(x),n(\gamma)(x))|\\
			\leq |\nabla G|^2(x)+|G(x)|^2+O_x^1(|f|^2).
		\end{multline}
		Setting $G=|\nabla^k h|^2g_0$ in this inequality we deduce that for  $k=0,1,2.$
		\begin{equation}\label{base.point}
			|\nabla^k h|(y)=O^{k+1}_x(|h|)+O_x^1(|f|).
		\end{equation}
		Setting  $G=L(h)$ in \eqref{taylor.epn} and using the fact that $L$ is a second order differential operator we deduce
		\begin{multline}\label{expn.Lh}
			|L(h)_x(n(\gamma)(x),n(\gamma)(x))-L(h)_y(n(\gamma,f)(y),n(\gamma,f)(y))|\\
			=O_x^3(|h|^2)+ O_x^1(|f|^2).
		\end{multline}
		Let $\nu_g(\gamma,f)$ denote the unit normal vector field along $\Sigma(\gamma,f)$ with respect to $g=g_0+h$ so that $\nu_{g_0}(\gamma,f)=n(\gamma,f)$ and $\nu_g$ depends smoothly on its parameters. Then
		$$|\nu_g(\gamma,f)(y)-n(\gamma,f)(y)|=O_y^0(|h|)=O_x^1(|h|)+ O_x^1(|f|),$$
		where in the last identity we used \eqref{base.point}. Using this identity and \eqref{base.point}  we have
		\begin{multline*}
			L(h)_y(\nu_g(\gamma,f)(y),\nu_g(\gamma,f)(y))\\
			=L(h)_y(n(\gamma,f)(y),n(\gamma,f)(y))+O_x^3(|h|^2)+ O_x^1(|f|^2).
		\end{multline*}
		Combining with \eqref{expn.Lh} we deduce
		\begin{multline*}
			L(h)_y(\nu_g(\gamma,f)(y),\nu_g(\gamma,f)(y))\\
			=L(h)_x(n(\gamma)(x),n(\gamma)(x))+O_x^3(|h|^2)+ O_x^1(|f|^2).
		\end{multline*}
		Using Taylor's expansion  we have  that for every $y\in M$, every unit vector field $Y\in T_yM$ and $g_0+h\in \mathcal U$
		$$|\mathring{Ric}(g_0+h)_y(Y,Y)-L(h)_y(Y,Y)|=O_y^2(|h|^2)=O_x^3(|h|^2)+O_x^1(|f|^2).$$
		Therefore
		\begin{multline*}
			\mathring{Ric}(g_0+h)_y(\nu_g(\gamma,f)(y),\nu_g(\gamma,f)(y))\\
			=L(h)_y(n(\gamma)(x),n(\gamma)(x))+O_x^3(|h|^2)+ O_x^1(|f|^2).
		\end{multline*}
		The first statement in the proposition follows from choosing $f=f_\gamma$ in the identity above.
		
		We now prove the statement regarding $|Jac_g F_{\gamma}|$. With $(\gamma,f)\in X$,  denote by $|Jac_g F(\gamma,f)|(g)$ the Jacobian of $F(\gamma,f):(D(\gamma),g_0)\rightarrow (\Sigma(\gamma,f),g)$. With $g=g_0+h\in \mathcal U$ and $y=F(\gamma,f)(x)$ we have
		\begin{align*}
			|Jac_g F(\gamma,f)|(x) & =|Jac_{g_0}F(\gamma,f)|(x)+ O^0_y(|h|)\\
			&=|Jac_{g_0}F(\gamma,f)|(x)+ O^1_x(|h|)+O^1_x(|f|)\\
			&=|Jac_{g_0}F(\gamma,0)|(x)+O^1_x(|h|)+O^1_x(|f|)\\
			&=1+O^1_x(|h|)+O^1_x(|f|).
		\end{align*}

	\end{proof}

	\subsection{Stability of Ricci flow}
	
	We assume we have a solution to normalized Ricci flow $(\bar g_t)_{t\geq 0}$ \eqref{RF} and a smooth family of diffeomorphisms $\{\Phi_t\}_{t\geq 0}$  converging strongly to some diffeomorphism $\Phi_{\infty}$so that  $g_t:=\Phi_t^*\bar g_t$ solves the DeTurck-modified Ricci flow (which is strictly parabolic) and converges to the hyperbolic metric $g_0$ as $t\to\infty$.

	{Recall the operator $\mathcal A$ defined in \eqref{deturck}.} It has discrete spectrum {$1\leq \lambda_1\leq\lambda_2\leq \ldots $} 

	\begin{prop}\label{ricci-stable} 
		Set $h_t:=\Phi^*_t\bar g_t-g_0$, $t\geq 0$. Then 
		
		\begin{itemize}
			\item $|h_t|_{C^4}\leq O(e^{-2t/3})$;
			\item  {For all $k\in\N$ the tensor} $e^{t}h_t$ converges in $W^{k,2}$ as $t\to\infty$ to $\bar h$, where $\bar h$ is  a smooth $2$-tensor with $\mathcal A(\bar h)=-\bar h$.
		\end{itemize}
	\end{prop} 
	\begin{proof}
		In \cite{knopf} it is  shown that  $|h_t|_{C^{2}}\leq O(e^{-2t/3})$. Standard estimates (similar to \cite[Lemma 5.3]{haslhofer}) show that for all $k\in\N$, $|Ric(\bar g_t)+2\bar g_t|_{C^{k}}\leq O^{{k}}(e^{-2t/3})$ and thus $|\bar g_t-\bar g_{+\infty}|_{C^{{k}}}\leq O^{{k}}(e^{-2t/3})$ as well.  The diffeomorphisms $\{\Phi_t\}_{t\geq 0}$ depend only on $g_0$ and $\{\bar g_t\}_{t\geq 0}$, and one can check that for all $k\in\N$, $|\Phi_t-\Phi_{+\infty}|_{C^{k}}\leq O^{{k}}(e^{-2t/3})$. Thus  $|h_t|_{C^{{k}}}\leq O^{{k}}(e^{-2t/3})$.
		
		In \cite{knopf} it is  also show that the tensors $h_t$  satisfy  an equation of the form
		$$\frac{dh_t}{dt}=\mathcal A(h_t)+ Q_t,$$
		where $Q_t$ is a non-linear term depending on $g_0, h_t, \nabla h_t,\nabla^2h_t$.  The important property we need is that $|Q_t|_{C^k}\leq O^{{k}}(|h_t|_{C^{k+2}}^2)$. From here on we consider $Q_t$ as being a fixed non-homogeneous term where $|Q_t|_{C^{{k}}}\leq O^{{k}}(e^{-4t/3})$.
		
		Consider an  $L^2$-orthonormal  basis $\{u_{j}\}_{{j}\in\N}$ for the space of symmetric $2$-tensors on $M$ made of eigentensors for $\mathcal A$. Necessarily
		
		$$h_t=\sum_{{j}=0}^{\infty}\left(e^{-\lambda_{j} t}\langle  h_0,u_{j}\rangle +\int_0^te^{\lambda_{j}(\tau-t)}\langle Q_\tau,u_{j}\rangle d\tau\right) u_{j}.$$
		{Set $\eta_t:=h_t-f(t),$ where
			$$f(t):=\sum_{\lambda_{j}=1}\left(e^{-t}\langle  h_0,u_{j}\rangle +\int_0^te^{(\tau-t)}\langle Q_\tau,u_{j}\rangle d\tau\right) u_{j}.$$
			With {$\eta_t^0:=\eta_t$ and $k\in\N$, set  $\eta^k_t=\mathcal A(\eta^{k-1}_t)$. Because $\eta^k_t$} is orthogonal to the $1$-eigentensors there is some $\delta$ so that
			$$\int_{M}\mathcal A({\eta^{k}_t})\eta^k_tdV_{g_0}\leq-(1+2\delta)\int_{M}({\eta^k_t})^2dV_{g_0}.$$
			{Moreover we have
			$\partial_t \eta^k_t=\mathcal A(\eta^k_t)+ O^k(e^{-4t/3}),$}
			which when combined with the inequality above  implies that
			\begin{multline*}
				\frac{d}{dt}\int_{M}(\eta^k_t)^2dV_{g_0}\leq -2(1+2\delta) \int_{M}(\eta^k_t)^2dV_{g_0}+O^k(e^{-4t/3}) \int_{M}\eta^k_tdV_{g_0}\\
				\leq -2(1+2\delta) \int_{M}(\eta^k_t)^2dV_{g_0}+O^k(e^{-4t/3})\left( \int_{M}(\eta^k_t)^2dV_{g_0}\right)^{1/2}\\
				\leq -2(1+\delta) \int_{M}(\eta^k_t)^2dV_{g_0}+O^k(e^{-8t/3}).
			\end{multline*}
			{Hence $\int_{M}(e^{t}\eta^k_t)^2dV_{g_0}\to 0$ as $t\to\infty$. } {From the fact that $\mathcal \eta^{k+1}_t-\Delta_{g_0}\eta^k_t$ is linear in $\eta^k_t$ we obtain that for all $k\in\N_0$ we have} 
			$$\int_{M}(e^{t}\eta^k_t)^2+(e^{t}\Delta_{g_0}\eta^k_t)^2dV_{g_0}\to 0\mbox{ as }t\to\infty.$$
			{This implies that, for all $k\in \N$, $e^{t}\eta_t$ tends to zero in $W^{k,2}$ as $t\to\infty$.} The result follows because the $1$-eigentensor $e^tf(t)$ converges  {smoothly}  as $t\to\infty$ to some smooth tensor $\bar h$.}
		\end{proof}
	
	\subsection{A theorem of Calegari-Marques-Neves} 
	Consider a sequence $\phi_i\in\mathcal F(M,1/i)$ equivariant  with respect to a representation of a Fuchsian subgroup of $\text{PSL}(2,\R)$ in $\pi_1(M)<\text{PSL}(2,\C)$. Let $G_i<\pi_1(M)$ be the image of that representation. The group $G_i$  preserves $D(\partial\phi_i)\subset \H^3$ and recall that ${\bf D}(\partial \phi_i)={D}(\partial \phi_i)/ G_i.$ We assume
	
	\begin{itemize}
		\item $\Omega_*\delta_{\phi_i}$ converges weakly to a measure $\nu$ on $F(M)$ as $i\to\infty$, where the measure $\nu$ is such that  so that $\nu(O)>0$ for every open set $O$;
		\medskip
		
		\item there is  a sequence of  immersed surfaces $\Sigma_i\subset M$  homotopic to ${\bf D}(\partial \phi_i)$ and $f_i\in C^0(\Sigma_i)$ so that  
		$$\lim_{i\to\infty}\frac{1}{{\rm area}_{g_0}(\Sigma_i)}\int_{\Sigma_i}|f_i|dA_{g_0}=0.$$
	\end{itemize}
	The Hausdorff distance between two sets in $\H^3$ is denoted by $d_H$.
	
	The following theorem corresponds to  Theorem 6.1 of \cite{cal-marques-neves}, where it is assumed that   $\Sigma_i$ is  area-minimizing with respect to some metric. An inspection of the proof shows that one only needs  the areas of $\Sigma_i$ and ${\bf D}(\partial \phi_i)$ to be  comparable and their universal covers to be at a uniform Hausdorff distance from each other.
	\begin{thm}\label{CMN.thm} Assume the existence of $C>0$ and, for all $i\in\N$, a covering $\Omega_i\subset\H^3$ of $\Sigma_i$ so that
		$${\rm area}_{g_0}(\Sigma_i)\leq C{\rm area}_{g_0}({\bf D}(\partial \phi_i))\quad\mbox{and}\quad
		C^{-1}\leq d_H(\Omega_i,D(\partial \phi_i))\leq C.$$
		For every $\gamma\in \mathcal C_0$ there is $\eta_i\in \pi_1(M)<\text{Isom}\,(\H^3)$ such that, after passing to a subsequence, $\eta_i(D(\partial \phi_i))$ converges on compact sets to $D(\gamma)$ and 
		$$\lim_{i\to\infty}\int_{\eta_i(\Omega_i)\cap B_R(0)}|f_i\circ\eta_i^{-1}|dA_{g_0}=0\quad\mbox{for all }R>0.$$
	\end{thm}

	\bibliographystyle{amsbook}

\end{document}